\documentclass[12pt,reqno]{amsart}
 \usepackage{graphicx}
 \usepackage[utf8]{inputenc}
 \usepackage[T1]{fontenc}
 \usepackage{lmodern}
 \usepackage[normalem]{ulem}
 \usepackage{verbatim}
 \usepackage{bbm}
 \usepackage{stmaryrd}
 \usepackage{amsmath}
 \usepackage{amssymb}
 \usepackage{dsfont}
\usepackage{amsthm}

\usepackage{xcolor}

\usepackage{amsmath}%
\usepackage{amsfonts}%
\usepackage{amssymb}%
\usepackage{graphicx}
\usepackage{dsfont}

\theoremstyle{plain} 
\newtheorem{theorem}{Theorem}[section]
\newtheorem*{theorem*}{Theorem}
\newtheorem{lemma}[theorem]{Lemma}
\newtheorem*{lemma*}{Lemma}
\newtheorem{corollary}[theorem]{Corollary}
\newtheorem*{corollary*}{Corollary}

\newtheorem*{proposition*}{Proposition}
\newtheorem{assumption}[theorem]{Assumption}
\newtheorem*{assumption*}{Assumption}

\newtheorem*{definition*}{Definition}

\newtheorem*{example*}{Example}
\newtheorem{remark}[theorem]{Remark}

\newtheorem*{remark*}{Remark}
\newtheorem*{remarks*}{Remarks}

\newcommand{\E}{\mathbb{E}}
\newcommand{\N}{\mathbb{N}}

\newcommand{\R}{\mathbb{R}}

\newcommand{\T}{\mathbb{T}}

\renewcommand{\P}{\mathbb{P}}

\DeclareMathOperator{\tr}{Tr}

\renewcommand{\leq}{\leqslant}
\renewcommand{\geq}{\geqslant}
\renewcommand{\epsilon}{\varepsilon}

\def\tr{{\rm Tr}}
\def\R{{\mathbb R}}
\def\rank{{\rm rank}}
\def\T{{\mathbb T}}
\title{Large deviations for Gibbs ensembles of the classical Toda chain}

\author{Alice Guionnet}
         \address{Universit\'e de Lyon, ENSL, CNRS,  France}
\email{Alice.Guionnet@ens-lyon.fr}

		\author{ Ronan ~Memin}
           \address[Ronan Memin]{Universit\'e de Lyon, ENSL, CNRS,  France}
      \email{Ronan.Memin@ens.fr}
\thanks{This project   has received funding from the European Research Council (ERC) under the European Union
Horizon 2020 research and innovation program (grant agreement No. 884584).}

\begin{document}
\maketitle

{\bf Abstract} We prove large deviation principles for the distribution of the empirical measure of the eigenvalues of Lax matrices  following  the Generalized Gibbs ensembles of the classical Toda chain introduced in \cite{Spohn1}. We deduce the almost sure convergence  of this empirical measure towards a  limit which we describe  in terms of the limiting  empirical measure of Beta-ensembles. Our results apply to general smooth potentials. 

\tableofcontents
\section{Introduction}
 In a breakthrough paper \cite{Spohn1}, Herbert Spohn introduced the generalized Gibbs  ensembles of the classical Toda chain as invariant measures of the dynamics of the classical Toda lattice. He  analyzes them by comparing the Toda Lax matrices for  these Generalized Gibbs ensembles with  Dumitriu-Edelman tri-diagonal representations of $\beta$-ensembles. Thanks to this beautiful comparison,  \cite{Spohn1}  showed that
the empirical measure of the eigenvalues of   Toda Lax matrices converges towards a probability measure related with the equilibrium measure for $\beta$-ensembles.
   One of the key tools of Herbert Spohn analysis is the use of transfer matrices, which  are restricted to polynomial potentials. We refer the interested reader to subsequent developments in \cite{Spohn2,Spohn3,Spohn4} and \cite{GravaMazzuca} where  the transfer matrix approach is used in the similar context of the so-called Ablowitz-Laddik lattice.

The  main goal of this article is  to generalize some  of  the results of \cite{Spohn1} by using large deviations theory, which allows  to consider more general potentials. More precisely, we will show the convergence of the free energy and of the empirical measure of the eigenvalues of Toda Lax matrices following these Generalized Gibbs ensembles. Moreover, we will express the limits in terms of the well known $\beta$-ensembles. 
Indeed, a  key tool is again to compare the Toda Lax matrices  with  Dumitriu-Edelman tri-diagonal representations of $\beta$-ensembles. 
 Moreover, we will derive large deviation principles  for the empirical measure of the eigenvalues of  tri-diagonal matrices with more general coefficients. However, in this generality, the rate functions and the limits will not be explicit as the comparison with $\beta$-ensembles is not possible.

More precisely, the Hamiltonian of the Toda chain on sites $j=1,\ldots, N$ is given by
$$H=\sum_{j=1}^N (\frac{1}{2}p_j^2 +e^{-r_j}),\quad r_j=q_{j+1}-q_j$$
with the periodic conditions $q_{N+j}=q_j+cN$ for some real constant $c$. The equations of motion are then given by
\begin{equation}\label{dyn}\frac{d}{dt} q_j=p_j,\qquad \frac{d}{dt} p_j=e^{-r_{j-1}}-e^{-r_j}\,.\end{equation}
Let $L_N$ be the Lax matrix given by the $N\times N$  tri-diagonal matrix with entries 
\begin{equation}
\label{def matrice L}
(L_N)_{j,j} =p_j\text{ and  }(L_N)_{j,j+1}= (L_N)_{j+1,j} = e^{-r_j/2}
\end{equation}
with periodic boundary conditions $ (L_N)_{1,N} = (L_N)_{N+1,N} $ and  $(L_N)_{N,1}=(L_N)_{N,N+1}$, then for all integer number $p$,
$$Q^{p}_N=\tr (L_N^p)$$
is conserved by the dynamics \eqref{dyn} as well as $\sum_{i=1}^Nr_i$. It is therefore natural to consider that the finite $N$ Toda chain is distributed according to the  Gibbs measure with density $e^{-\tr(W(L_N)) -P\sum r_{i} }$ with respect to
Lebesgue measure. Here, 
$P>0$ controls the pressure of the chain and $W$ is a potential to be chosen later, which can be a polynomial or a general measurable function from $\R$ into $\R$. We will assume it goes to infinity faster than $x^2$: namely there exists $c>0$ and a finite constant $C$ such that for {all $x\in \R$}
\begin{equation}\label{ass1}
W(x)\ge cx^2+C\,.
\end{equation}
This assumption is used to compare our distribution to the case where $W(x)=c x^2$ in which case the entries of the Lax matrix $L_N$ are independent. 
We can without loss of generality assume $c=\frac{1}{2}$ up to rescaling and therefore put 
\begin{equation}\label{defW}W(x)=\frac{1}{2} x^2 +V(x).\end{equation}
In the following we will denote 
\begin{equation}\label{Gibbs}
d\T_N^{V,P}(p,r)=\frac{1}{\mathbb Z_{N,\T}^{V,P}} \exp\{-\tr(V(L_N))-\frac{1}{2}\tr(L_{N}^{2})\} \prod _{i=1}^N e^{-Pr_i}dr_i dp_i
\end{equation}
where $\mathbb Z_{N,\T}^{V,P}$ is the partition function of the Toda Gibbs measure :
\begin{equation}\label{pp}\mathbb Z_{N,\T}^{V,P}=\int \exp\{-\tr(V(L_N))-\frac{1}{2}\tr(L_{N}^{2})\} \prod _{i=1}^N e^{-Pr_i}dr_i dp_i\,.\end{equation}
We denote
in short  $\T_N^{P}$ for $\T_N^{0 ,P}$.
Our goal in this article is to study the empirical measure of the eigenvalues $\lambda_N\le\cdots\le \lambda_1$ of $L_N$
 denoted by 
$$\hat \mu_{L_N}=\frac{1}{N}\sum_{i=1}^N \delta_{\lambda_i}\,.$$
We shall call $\hat\mu_{L_N}$ the {empirical} measure of $L_N$, or the empirical density of states of the Lax matrix following \cite{Spohn1}.
Our main result is  a large deviations principle for the distribution of $\hat\mu_{L_N}$ under $d\T_N^{V,P}$, from which we deduce the almost  sure convergence of $\hat\mu_{L_N}$ under $d\T_N^{V,P}$.
\begin{theorem}\label{theoldp}
Let $P>0$ and assume that $V$ is continuous. Assume that either $V$ is uniformly bounded or there exists $k\in\N^{*}$ such that
\begin{equation}\label{condi}\lim_{|x|\rightarrow \infty} \frac{V(x)}{x^{2k}}=a\,,\end{equation}
with $a>0$ if $k>1$ and $a>-1/2$ if $k=1$. 
Then:
\begin{enumerate}
\item The law of $\hat\mu_{L_N}$ under $\T_N^{V,P}$ satisfies a large deviation principle in the scale $N$ with good rate function $T_P^V$.
\item  $T_P^V$  achieves its minimal value at a unique probability measure $\nu_P^V$.
\item As a consequence $\hat \mu_{L_N}$ converges almost surely and in $L^1$ towards $\nu_P^V$.
\end{enumerate}
\end{theorem}
$\nu^{V}_{P}$ corresponds to the density of states of the Lax matrix in  \cite{Spohn1}. 
Moreover, following 
 \cite{Spohn1}, we can identify the equilibrium measure 
 $\nu_P^V$ using the equilibrium measure for Coulomb gases in dimension one 
 at temperature of order of the {number of particles}. More precisely,
 for a probability measure $\mu$ on the real line, we define the function $f_{P}^{V}$ by 
 $$f_P^V(\mu)=\frac{1}{2} \int\left(\frac{1}{2}(x^{2}+y^{2})+V(x)+V(y)-2P\ln |x-y|\right)d\mu(x)d\mu(y) +\int \ln\frac{d\mu}{dx} d\mu(x)  $$
if $\mu\ll dx$, whereas $f_P^V$ is infinite otherwise.  $f_P^V$ achieves its minimal value at a unique probability measure $\mu_P^V\ll dx$ which satisfies the non-linear equation
\begin{equation}\label{carmu} \frac{1}{2} x^{2}+V(x)- 2P\int\ln |x-y|d\mu_P^V(y)+\ln\frac{d\mu_P^V}{dx} =\lambda_P^V\qquad   a.s \end{equation}
where $\lambda_P^V$ is a finite constant. We show  in section \ref{Coulomb} that $\mu_P^V$ is absolutely continuous with respect to Lebesgue measure and that its depends smoothly on the parameter $P$.  In Lemma \ref{convex}, we show it is in fact differentiable in $P$.
We then show that 
\begin{theorem}Let $P$ be a positive real number. Then,
for any bounded continuous function $f$ on the real line,
$$\int f(x) d\nu_P^V(x)=\partial_P (P \int f(x)d\mu_P^V(x))$$
\end{theorem}
This result was sketched in \cite{Spohn1} when $V$ is a polynomial, through the transfer operator technique. In \cite{mazzuca2022mean}, the author establishes it for a quadratic potential, through the analysis of the moments of the empirical measure of the Lax matrix.

Our strategy is to prove first a large deviation principle in the case when $V$ vanishes: then, $L_N$ has  independent entries (modulo the symmetry constraint) under $\T^{P}_N$. We then derive  large deviation principles for more general bounded continuous potentials by using Varadhan's Lemma, see section \ref{quad}. \\ Indeed, in the case where $V$ vanishes, 
the random variables  $(p_j,r_j)_{1\leq j \leq N}$ are independent, $(L_N)_{j,j}$ are standard Gaussian  $N(0,1)$ variables and  $\sqrt{2}(L_N)_{j,j+1}$ follows a  $\chi_{2P}$ distribution with density with respect to Lebesgue measure given by 
\begin{equation}
\label{loi du chi}
\chi_{2P}(x) = \frac{2^{1-P}x^{2P-1}e^{-x^2/2}}{\Gamma(P)}\mathbf{1}_{x>0}.
\end{equation}  
	
	The central observation is that we can compare this matrix to the tri-diagonal matrix $C^{\beta}_{N}$ introduced by 
Dumitriu and Edelman  \cite{dued}. This is the symmetric matrix with independent (up to symmetry) entries whose diagonal elements are independent  standard Gaussians variables, and off diagonal elements so that $\sqrt{2} C_N^{\beta}({j,j+1})$ follow a $\chi$ distribution with parameter $\beta(N-j)$. When $\beta=2P/N$, the matrix is therefore similar to $L_{N}$ except that the parameters of the off-diagonal entries vary linearly. The key point is that the law of the eigenvalues of $C^{\beta}_{N}$ is explicit and given by the 
 $\beta$-ensemble, see  Section \ref{Coulomb}.
This comparison allows to compare  the free energy, the rate function {and} the equilibrium measure of the Toda chain with those of Coulomb gases in section \ref{Coulomb}. In section \ref{general}, we study the case of general potentials. The proof is nearly independent from the quadratic case, but requires additional arguments in particular because the eigenvalues of the Toda matrix are not simple functions of the empirical measure of the entries. {Note that the proof given in section \ref{general} also applies to the case where $V$ is bounded. We nevertheless choose to give a separate proof, dedicated to this case: the computations being simpler, the core of the proof seems more accessible and introduces ideas we re-use in the case where $V$ is unbounded.}

Moreover, our result allows to derive large deviation principles for the empirical measure of the tri-diagonal matrices with independent standard Gaussian entries on the diagonal and independent chi distributed variables 
with  general parameters  profile on the off-diagonal. Namely let $L_N^\sigma$ be a tri-diagonal symmetric matrix with independent Gaussian variables on the diagonal and independent variables $\sqrt{2} L_{N}^{\sigma}(j,j+1)$ chi distributed with parameter $\sigma(\frac{i}{N}), 1\le i\le N$. Let  $\T_{N}^{V,\sigma}$ be the distribution with density $e^{-\tr(V(L_{N}^{\sigma}))}/Z$ with respect to the distribution of $L_N^\sigma$.
\begin{theorem}\label{generalization}
Assume that $V$ is continuous and satisfies \eqref{condi}. Then, if $\sigma$ is bounded continuous,  
\begin{enumerate}
\item the law of $\hat\mu_{L_N^\sigma}$ under $\T_{N}^{V,\sigma}$ satisfies a large deviation principle in the scale $N$ with good rate function $T^{V}_{\sigma}$,
\item  $T^{V}_{ \sigma}$  achieves its minimal value at a unique probability measure $\nu_\sigma^{V}=\int_0^1 \nu_{\sigma(P)}^V dP$,
\item As a consequence, $\hat \mu_{L_N^\sigma}$ converges almost surely and in $L^1$ towards $\nu_\sigma^{V}$.
\end{enumerate}
 \end{theorem}

{\bf Acknowledgments :} We are very grateful to Herbert Spohn for asking us to investigate the convergence of the density of states for general potentials $V$ and many fruitful discussions that followed. We would also like to thank David García-Zelada for showing us how to derive Theorem \ref{David} from \cite{Garcia}.  We thank an anonymous referee for helping us to improve the presentation of our results.

\section{Large deviation principles for tri-diagonal matrices}\label{quad}
In this section, we consider a tri-diagonal  matrix $M_N$ with entries 
\begin{equation}
\label{jacobi}
(M_N)_{j,j} =a_j\text{ and  }(M_N)_{j,j+1}= (M_N)_{j+1,j} = b_j
\end{equation}
with periodic boundary conditions, the  random variables  $(a_i,b_i)_{1\le i\le N}$ being iid, with $(a_1,b_1)$ with law $Q_a\otimes Q_b$ on $\R^2$. We denote by $\hat \mu_{M_{N}}$ the empirical measure of the eigenvalues of $M_N$ and prove the existence of a large deviation principle for the distribution of $\hat\mu_{M_N}$.
In \cite[Theorem 4.2]{zhang}, the author proves a large deviation principle for the empirical moments $\hat\mu_{M_N}(x^k)$ by noticing that
$$\hat\mu_{M_N}(x^k)=\frac{1}{N}\sum_{i=1}^N f_k(a_j,b_j, |i-j|\le k)$$
where  $f_k(a_j,b_j, |i-j|\le k)=(M_N^k)_{ii}$ is an homogeneous  polynomial of degree $k$ in the entries $a_j,b_j, |i-j|\le k$. Noting that $f_{k}$ does not depend on $i$, 
one can use the large deviation principle for Markov chains (or $k$-dependent large deviation principle), see e.g  \cite[Theorem 3.1.2 or Section 6.5.2]{DZ}, as well as the contraction principle, to deduce a large deviation principle for the distribution of the  empirical moments $\{\hat\mu_{M_N}(x^k), k\le p\}$. 
This could be used to deduce the existence of a large deviation principle for $\hat\mu_{M_{N}}$ for the weak topology after approximations, but the rate function would not be particularly explicit. We prefer to develop a more straightforward sub-additivity argument and prove separately the existence of a weak large deviation principle and exponential tightness, see e.g \cite[Lemma 1.2.18]{DZ}. 
\subsection{Exponential tightness}
In this  section we assume that
\begin{assumption}\label{ass} 
There exists $\gamma>0$ such that
$$D_\gamma:=\int e^{\gamma x^2}dQ_a(x)\times \int e^{\gamma y^2}dQ_b(y)<\infty\,.$$
\end{assumption} 
We equip the set of probability measures on the real line $\mathcal P(\mathbb R)$ with the weak topology. 
We then show that 
\begin{lemma}\label{exptight}
If $(a_{j},b_{j})_{1\le j\le N}$ are iid with law $Q_{a}\otimes Q_{b}$ satisfying  Assumption \ref{ass}, the sequence $(\hat\mu_{M_{N}})_{N\ge 0}$ is exponentially tight, namely for each $L\ge 0$ there exists a compact set $K_L$ ($K_L=\{\mu\in \mathcal P(\mathbb R):\int x^2 d\mu(x)\le  \frac{2}{\gamma}(L+\ln  D_\gamma)\}$ with $\gamma$ as in  Assumption \ref{ass}) such that
\begin{equation}
\label{def tension exponentielle}
\limsup_N \frac{1}{N} \ln  \P( \hat{\mu}_{M_N} \in K_{L}^c ) < -L.
\end{equation}
\end{lemma}
\begin{proof}
For $N\geq 1$, notice that
\begin{eqnarray}
\label{tension exponentielle}
\int x^2 \text{d}\hat{\mu}_{M_{N}}(x)&=&\frac{1}{N}\tr(M_N^2)\nonumber\\
&=&\frac{1}{N}\sum_{j=1}^N \left((M_N)_{j,j}\right)^2 +\frac{1}{N}\sum_{j=1}^N \left(\sqrt{2}(M_N)_{j,j+1}\right)^2.
\end{eqnarray}
As a consequence, Tchebychev's inequality implies that, for any $\gamma>0$, 
\begin{eqnarray*}
\P\left( \int x^2 \text{d}\hat{\mu}_{M_N}(x)>K\right)&\le & e^{-\frac{1}{2}\gamma NK} \E[e^{\frac{1}{2}N \gamma \int  x^2 \text{d}\hat{\mu}_{M_N}(x)}]\le  e^{-\frac{1}{2}\gamma NK} D_\gamma^{N}\,.\end{eqnarray*}
The conclusion follows by taking $K=\frac{2}{\gamma}(L+\ln  D_\gamma)$.

\end{proof}
\subsection{Weak large deviation principle} 
We next establish a weak large deviation principle, based on the general ideas developed in \cite[ Lemma 6.1.7]{DZ}.
To this end, we use the  following distance on $\mathcal P(\R)$:
\begin{equation}
\label{distance BV}
d(\mu,\nu) = \sup_{\|f\|_\text{BV} \leq 1, |f|_\text{Lip}\leq 1} \left\lbrace \left| \int_\R f(x) d\mu(x) - \int_\R f (x) d\nu(x) \right| \right\rbrace,
\end{equation}
where $\|f\|_\text{BV}$  is the total variation norm of $f$
given by
$$
\|f\|_\text{BV}= \sup \sum_{k\in \N} |f(x_{k+1})-f(x_k)|,
$$
where the supremum holds over all increasing sequences 
$(x_k)_{k\in \N}\in \mathbb R^{\N}$. $\|f\|_L$ is the Lipschitz norm of $f$. If $f$ is continuously differentiable  and we put without loss of generality $f(0)=0$, $\|f\|_{BV}=\int_{-\infty}^{+\infty} |f'(y)|dy$ and $\|f\|_L=\|f'\|_\infty$. 
The distance $d$  is smaller than the Wasserstein distance where one takes the supremum over all functions whose $L^\infty$ and Lipschitz norms are bounded by one, and is easily seen to be as well compatible with the weak topology.
Then, we shall prove that if $ B_\mu(\delta)=\{\nu\in \mathcal P(\R): d(\mu,\nu)<\delta\}$ denotes the open ball with radius $\delta$ centered at $\mu$, we have : 
\begin{lemma}\label{wldp}
For any $\mu$ in $\mathcal P(\R)$,
there exists a limit
\begin{equation}
\label{PGD faible}
\lim_{\delta \to 0} \liminf_N \frac{1}{N}\ln  \P\left( \hat{\mu}_{M_N} \in B_\mu(\delta)\right) = \lim_{\delta \to 0} \limsup_N \frac{1}{N}\ln  \P\left( \hat{\mu}_{M_N} \in B_\mu(\delta)\right).
\end{equation}
We denote this limit by $-J_M(\mu)$.
\end{lemma}
\begin{proof}
The advantage of the distance $d$ is the following control:  For any symmetric $N\times N$ matrices 
$A$ and  $B$  with empirical measures of eigenvalues  $\hat{\mu}_A$ and  $\hat{\mu}_B$, we have:
\begin{equation}
\label{super inegalite}
d(\hat{\mu}_A,\hat{\mu}_B)\leq \min\left\{ \frac{\rank(A-B)}{N},\frac{1}{N}\sum_{i,j}|A(i,j)-B(i,j)| \right\}.
\end{equation}
Indeed, for any function $f$ with bounded variation we have thanks to Weyl interlacing property, see e.g. \cite[(1.17)]{GuioFlour}, 
\begin{equation}\label{bound12}\left|\int fd\hat\mu_A-\int fd\hat\mu_B\right|\le \frac{1}{N}\rank({A-B})\,.\end{equation}
Moreover, one can check that, if $f$ is continuously differentiable, we have
\begin{eqnarray*}
\int fd\hat\mu_A-\int fd\hat\mu_B&=&\int_0^1 \frac{1}{N}\tr\left( (A-B)f'(\alpha A+(1-\alpha)B)\right)d\alpha \\
&=& \int_0^1 \left(\frac{1}{N}\sum_{i,j=1}^N  (A-B)_{ij} f'(\alpha A+(1-\alpha)B)_{ji} \right)d\alpha\end{eqnarray*}
which implies since for all indices $i,j$,  $|f'(\alpha A+(1-\alpha)B)_{ji}|\le\|f'\|_\infty$ that
\begin{equation}\label{bound2}\left|\int fd\hat\mu_A-\int fd\hat\mu_B\right|\le \|f'\|_\infty \frac{1}{N}\sum_{i,j=1}^N |(A-B)_{ij}|\,.\end{equation}
Since continuously differentiable  functions with bounded $L^\infty$ norm are dense in Lipschitz functions, we deduce \eqref{super inegalite} from \eqref{bound12} and \eqref{bound2}. 
We are now ready to prove
Lemma \ref{wldp}.  To this end, we shall approximate our matrix $M_N$ by a diagonal block matrix with independent blocks. 
Let $q\geq 1$. For  $N \geq 1$ we  decompose $N=k_Nq +r_N$ with  $r_N \in \left\lbrace 0,\ldots, q-1\right\rbrace$ and set  $M_N=M_N^q + R_N^q$, where $M_N^q$ is the  diagonal block matrix

\begin{equation}
\label{matrice decoupe}
  M_N^q =
  \begin{bmatrix}
    M_q^1 & & &\\
    & \ddots & &\\
    & & M_q^{k_N} & \\
    & & & B
  \end{bmatrix}.
\end{equation}
Here,  for all 
 $i\in \left\lbrace 1,\ldots,k_N\right\rbrace$,  $M_q^i$ has the same distribution than $ M_q$ and $B$ the same distribution than $  M_{r_N}$. The matrices $M_q^i, 1\le i\le k_{N},$ are independent, and are independent from $B$. 
 $R_N^q$ is the self-adjoint matrix with null entries except  $R_N^q(1,N)=R_N^q(N,1)=b_N$, $R_{N}^{q}(k_{N}q+1,N)=R_{N}^{q}(N,k_{N} q+1)=-b_{N}$, 
  and those given, for $k\in \{1,\ldots, k_N\}$,  by $R_N^q(kq+1,kq)=R_N^q(kq,kq+1)=b_{kq}$,  $R_N^q((k-1)q+1,kq))=R_N^q(kq,(k-1)q+1)=-b_{kq}$. Therefore 
$\rank (R_N^q) \leq 2k_N +2 \leq 4k_N$. 
By (\ref{super inegalite}), we deduce that 
\begin{equation}
\label{inegalite R}
d(\hat{\mu}_{M_N},\hat{\mu}_{M_N^q})\leq \frac{4}{q}.
\end{equation}
Moreover, we can write
$\hat{\mu}_{M_N^q}$ as the sum
$$
\hat{\mu}_{M_N^q} = \sum_{i=1}^{k_N}\frac{q}{N} \hat{\mu}_{M^i_q} + \frac{r_N}{N}\hat{\mu}_{B}\,.
$$
Therefore, for any $\mu\in \mathcal{P}(\R)$ and $\delta >0$, we have 
\begin{align*}
\P\left( \hat{\mu}_{M^1_q} \in B_\mu(\delta) \right)^{k_N}\P\left( \hat{\mu}_{M_{r_N}} \in B_\mu(\delta) \right) &= \P\left(\forall\ i\in\{1,\ldots,k_N\},\  \hat{\mu}_{M^i_q} \in B_\mu(\delta),\ \hat{\mu}_{B} \in B_\mu(\delta)  \right) \\
&\leq  \P\left(\hat{\mu}_{M_N^q} \in B_\mu(\delta) \right) \\
&\leq  \P\left( \hat{\mu}_{M_N} \in B_\mu(\delta + \frac{4}{q}) \right),
\end{align*} where we used the convexity of balls and  (\ref{inegalite R}).  As a consequence, 
$$ u_N(\delta): = -\ln  \P\left( \hat{\mu}_{M_N} \in B_\mu(\delta) \right)$$
satisfies
$$
u_N(\delta + 4/q) \leq k_Nu_q(\delta)+u_{r_N}(\delta).
$$
It is easy (and classical) to deduce the convergence of  {$u_N(\delta)/N$ when $N$ goes to infinity, and then  $\delta$ goes to zero}.
Indeed let  $\delta>0$ be given and choose $q$ large enough so that  $\frac{4}{q}< \delta$. Then, since $\delta\rightarrow u_N(\delta)$ is decreasing and non-negative, we have:
\begin{equation}
\frac{u_N(2\delta)}{N} \leq \frac{u_N(\delta + 4/q)}{N} \leq \frac{u_q(\delta)}{q} + \frac{u_{r_N}(\delta)}{N}\,.
\end{equation}
Since  $\frac{u_{r_N}(\delta)}{N} \leq \frac{ \max_{1\leq i \leq q-1} u_i(\delta) }{N} $  goes to zero  when $N\to \infty$, we conclude that
$$
\limsup_N \frac{u_N(2\delta)}{N} \leq \frac{u_q(\delta)}{q}\,.
$$ Since this is true for all $q$ large enough, 
we get
$$
\limsup_N \frac{u_N(2\delta)}{N} \leq \liminf_N \frac{u_N(\delta)}{N}\,.
$$
Since the left and right hand sides decrease as $\delta$ goes to zero, we conclude that
$$
\lim_{\delta\rightarrow 0} \limsup_{N \rightarrow \infty}  -\frac{1}{N} \ln \P \left( \hat{\mu}_{M_N} \in B_\mu(\delta) \right)\le \lim_{\delta\rightarrow 0} \liminf_{N \rightarrow \infty}  -\frac{1}{N} \ln  \P\left( \hat{\mu}_{M_N} \in B_\mu(\delta) \right)\,,$$
and the conclusion follows.
\end{proof} 

\subsection{Full large deviation principle}
As a consequence of Lemmas \ref{exptight} and \ref{wldp},  we have by \cite[Theorem 1.2.18]{DZ} the following large deviation principle.
\begin{theorem}\label{ldpgen} Under Assumption \ref{ass}, the law of $\hat \mu_M$ satisfies a large deviation principle  in the scale $N$ with a good rate function $J_M$. Moreover, $J_M$ is convex.
In other words, 
\begin{itemize} \item $J_M:\mathcal P(\R)\rightarrow [0,+\infty]$ has compact level sets $\{\mu:J_M(\mu)\le L\}$ for all $L\ge 0$. Moreover, $J_M$ is convex.
\item For any closed set $F\subset \mathcal P(\R)$,
$$\limsup_{N\rightarrow\infty}\frac{1}{N}\ln  \P(\hat\mu_{M_N}\in F)\le -\inf_F J_M\,,$$
whereas for any open set $O\subset \mathcal P(\R)$
$$\liminf_{N\rightarrow\infty}\frac{1}{N}\ln  \P(\hat\mu_{M_N}\in O)\ge -\inf_O J_M\,.$$
\end{itemize}
\end{theorem}
\begin{proof}$J_M$ exists and is defined by Lemma \ref{wldp}. The lower semi-continuity of $J_M$ follows from \cite[Theorem 4.1.11]{DZ}.  We then deduce that the level sets of $J_M$ are compact by the exponential tightness, see \cite[Lemma 1.2.18 (b)]{DZ}.

In the spirit of \cite[Lemma  4.1.21]{DZ}, we show that $J_M$ is convex. Let  $\mu_1$, $\mu_2 \in \mathcal{P}(\R)$.  Since $\hat\mu_{M_{2N}}$ can be decomposed as the independent sum of $\hat\mu_{M_{N}}$ divided  by 2 plus an error term of  smaller than $4/N$ by \eqref{bound12},  we have for all $\delta_1,\delta_2>0$
\begin{equation}
\label{quasi convexite}
 \P\left(d(\hat{\mu}_{M_N},\mu_1)<\delta_1\right) \P\left(d(\hat{\mu}_{M_N},\mu_2)<\delta_2\right)  \leq  \P\left(d(\hat{\mu}_{M_{2N}},\frac{\mu_1+\mu_2}{2})<\delta_3\right).
\end{equation}
for any    $\delta_3\ge \frac{1}{2}(\delta_1+\delta_2)+\frac{4}{N}$. Taking the logarithm, dividing by $2N$ and  letting $N$ go to infinity, $\delta_1,\delta_2$ and then $\delta_3$ to zero,  we conclude that  
\begin{equation}
J_M\left(\frac{\mu_1 + \mu_2}{2}\right) \leq  \frac{1}{2}\bigg( J_M(\mu_1)+J_M(\mu_2) \bigg),
\end{equation}
{from which we deduce the convexity of $J_M$ as in \cite[Lemma  4.1.21]{DZ}.}\\
The second point, namely that a weak large deviation principle and exponential tightness implies a full large deviation principle,  is classical, see \cite[Lemma 1.2.18]{DZ}.
\end{proof}
\subsection{Large deviation principle for the Toda-Chain with quadratic potential}
{Recall that we denoted by $Q_a$ and $Q_b$ respectively the laws of the $a_i$'s and $b_i$'s, see \eqref{jacobi}.}
In the case of the Toda chain with Gaussian potential, that is $V=0$, with entries following $\T_N^{P}$, we take $Q_a$ to be the standard Gaussian law and $Q_b$ to be the chi distribution  $\sqrt{2}^{-1}\chi_{2P}$ given in \eqref{loi du chi}. We let $L_N(P)$  be the tridiagonal matrix whose entries follow $\T^{P}_N$.
These entries clearly satisfy Assumption \ref{ass} and therefore we have
\begin{corollary}\label{ldpcor} For any $P>0$, 
the law of $\hat\mu_{L_N(P)}$  satisfies a large deviation principle in the scale $N$ with good, convex, rate function {denoted by} $T_P$.
\end{corollary}

For further use, we show that
\begin{lemma}
	\label{lemme couplage}
	For each $\mu\in \mathcal{P}(\R)$, the map {$P\in (0,+\infty)\mapsto T_P(\mu)$} is lower semi-continuous.
\end{lemma}
\begin{proof}Let $P,h$ be positive real numbers.  We first couple 
the matrices 
$(L_N({P}), L_N({P}+h))_N$, where  $L_N({u})$ follows $\T_{N}^{u}$ for $u=P$ and $u=P+h$, in such a way that
 there exists a finite constant $c$ so that
\begin{equation}
	\label{inegalite expo}
	\P\left( d(\hat \mu_{L_N({P})},\hat\mu_{L_N({P}+h)})>\delta \right) \leq e^{N(c-\sqrt{-\ln (h)}\delta/2)}\,.
\end{equation} 

This coupling is done as follows: \\
$\bullet$ The diagonal coefficients are the same set of  standard independent Gaussian variables \\
$\bullet$ The coefficient below and above the diagonal  $X^i_u$, follow a $\sqrt{2}^{-1}\chi_{2 u}$ for $u={P}$ , $u=h$ and ${P}+h$. By definition of the $\chi$ distribution we can construct these variables so that almost surely
\begin{equation}\label{couple}X^i_{{P}+h} =\sqrt{ (X^i_{P})^2 + (X^i_h)^2 }\,.\end{equation}

This coupling allows by \eqref{super inegalite} to write 
$$
d(\hat \mu_{L_N({P})},\hat \mu_{L_N({P}+h)}))\leq \frac{2}{N}\sum_{i=1}^N | X^i_{{P}+h} - X^i_{P} | = \frac{2}{N}\sum_{i=1}^N (X^i_{{P}+h} - X^i_{P})\le \frac{2}{N}\sum_{i=1}^N X^i_h, 
$$
where we ultimately used that, for all $i\in \{1,\ldots,N\}$, $X^i_{{P}+h}\le X^{i}_{h}+X^{i}_{P}$ because $X^{i}_{h}X^{i}_{P}$ is non-negative and \eqref{couple} holds.
Equation \eqref{inegalite expo} follows by Tchebychev inequality since $\E[\exp\{\sqrt{\ln h^{-1}} X^i_h\}]$ is finite, see \eqref{mo}.
\eqref{inegalite expo} implies that $(\hat\mu_{L_N({P}+h)})_{N\ge 0}$  is an exponential approximation of $(\hat\mu_{L_N({P})})_{N\ge 0}$ when $h$ goes to zero. By \cite[Theorem 4.2.16 ]{DZ} , we deduce that  for any $\mu \in \mathcal{P}(\R)$, we have
$$
T_{{P}}(\mu) = \lim_{\delta \to 0} \liminf_{h\to 0} \inf_{B_\mu(\delta)} T_{{P}+h}.
$$
By monotonicity of the right hand side and the lower semi-continuity of $T_{P+h}$ we deduce that, see \cite[(4.1.2)]{DZ}, 
 $$
\lim_{\delta\to0}\inf_{B_\mu(\delta)} T_{{P}+h} = T_{{P}+h}(\mu),
$$
and therefore
$$
T_{P}(\mu) = \lim_{\delta \to 0} \liminf_{h\to 0} \inf_{B_\mu(\delta)} T_{{P}+h} \leq \liminf_{h\to 0} T_{{P}+h}(\mu),
$$
and so ${P}\mapsto T_{P}(\mu)$ is lower semi-continuous. 

\end{proof}
We shall also use later that Corollary \ref{ldpcor} gives a  large deviation principle for the empirical measure of the Toda chain with general bounded continuous potential.
\begin{corollary}\label{continuite minimiseur}
	Let $V$ be a bounded continuous function on the real line  and $P$ be a positive real number. Let   $L_N(P)$ be the tridiagonal matrix whose entries follow $\T^{V,P}_N$. Then:
	\begin{itemize}
		\item The law of $\hat\mu_{L_N(P)}$ satisfies a large deviation principle in the scale $N$ with convex good rate function given, for any $\mu\in \mathcal P(\mathbb R)$,
		$$T_P^V(\mu)=T_P(\mu)+\int Vd\mu-\inf_\nu\{T_P(\nu)+\int Vd\nu\}\,.$$
		\item 
		The set $M_P^V$ where $T_P^V$  achieves its minimum value is a compact convex subset of $\mathcal P(\mathbb R)$. It is continuous in the sense that for any $\varepsilon>0$, there exists $\delta_\varepsilon>0$ such that for all $\delta<\delta_\varepsilon$, any {$P,Q>0$} such that  for ${|P-Q|}\le \delta$,
		$$M_{P}^V\subset (M_{Q}^V)^\varepsilon$$
		where $A^\varepsilon=\{\mu: d(\mu, A)\le \varepsilon\}$.
	\end{itemize}
	
\end{corollary}
\begin{proof}
The first point is a direct consequence of 
 Varadhan's lemma since when $V$ is bounded continuous, $\mu \rightarrow \int V(x) d\mu(x)$ is also continuous. We hence need only to  prove the second point, that is the continuity of ${P}\in (0,+\infty)\mapsto M_{P}^V$.  Note that since $T_P^V$ is  a good rate function, $M_{P}^V$ is compact for all positive real number $P$.
We let $\T_N$  be the coupling of $L_N({P})$ and $L_N({Q})$ introduced in Lemma \ref{lemme couplage}. By definition, for ${R=P}$ and ${Q}$, $B$  a measurable subset of $\mathcal P(\mathbb R)$,  we have 
$$\T^{V,{R}}_N(\hat\mu_{L_N}\in B)=\frac{1}{Z_{N,\T}^{V,{R}}} \int 1_{\{\hat\mu_{L_N({R}) }\in  B\}} e^{-N\int V(x) d\hat\mu_{L_N({R})}(x)}d\T_N,$$
where we used the notation
$$Z_{N,\T}^{V,{R}} = \int e^{-N\int V(x)d\hat{\mu}_{L_N({R})}(x)}d\T^{{R}}_N. $$
Therefore, since $ ( (M_{Q}^V)^\varepsilon)^c$ is open, we can use {the large deviation principle for the empirical measure of $L_N(P)$, Corollary \ref{ldpcor},} to state that for any $\kappa>0$
\begin{eqnarray}
&&-\inf_{( (M_{Q}^V)^\varepsilon)^c}T_{P}^V\le \limsup_{N\rightarrow\infty}\frac{1}{N}\ln \frac{1}{Z_{N,\T}^{V,{P}}} \int_{\{d(\hat\mu_{L_N({P})}, M_{Q}^V)>\varepsilon\}} 
e^{-N\int V(x) d\hat\mu_{L_N({P})}(x)} d\T_N
\nonumber\\
&&\qquad \leq \max\{ \limsup_{N\rightarrow\infty}\frac{1}{N}\ln \frac{1}{Z_{N,\T}^{V,{P}}} \int_{\{d(\hat\mu_{L_N({P})}, M_{Q}^V)>\varepsilon\}\cap\{d(\hat\mu_{L_N({P})},\hat\mu_{L_N({Q})})\le \kappa \}}
e^{-N\int V(x) d\hat\mu_{L_N({P})}(x)} d\T_N,\nonumber \\
&&\qquad\qquad\qquad 2\|V\|_\infty +c-\sqrt{-\ln|{P}-{Q}|} \kappa/2\}\label{bb}
\end{eqnarray}
where we used \eqref{inegalite expo} and $Z_{N,\T}^{V,{P}}\ge e^{-N\|V\|_\infty}$.  We next remark that by Lemma \ref{exptight}, there exists a positive constant $c$ and a finite constant $C$ such that uniformly on $P$ in a compact set, if we denote by $K_L=\{\int x^2 d\mu(x)\le L\}$, 
$$\T_N^{P}\left(\hat\mu_{L_N}\in K_L^c\right)\le e^{-(cL+C)N}\,.$$
Hence, fixing some $L>0$, \eqref{bb} implies
\begin{align}
&-\inf_{( (M_{Q}^V)^\varepsilon)^c}T_{P}^V \leq \max\bigg\{ 2\|V\|_\infty +c-\sqrt{-\ln|{P}-{Q}|} \kappa/2, 2\|V\|_\infty -cL-C, \label{bb2}\\
&
\limsup_{N\rightarrow\infty}\frac{1}{N}\ln \frac{1}{Z_{N,\T}^{V,{P}}} \int {\bf 1}_{d(\hat\mu_{L_N({P})}, M_{Q}^V)>\varepsilon}{\bf 1}_{d(\hat\mu_{L_N({P})},\hat\mu_{L_N({Q})})\le \kappa }{\bf 1}_{ \hat\mu_{L_N({P})},\hat\mu_{L_N({Q})}\in K_L
}
e^{-N\int V(x) d\hat\mu_{L_N({P})}(x)} d\T_N\bigg\}.\nonumber
\end{align}
We next notice that $\int V(d\mu-d\nu)$  is bounded  by some $\epsilon_{V}^L(\kappa)$  going to zero as $\kappa$  does uniformly  on $\{d(\mu,\nu)\le\kappa\}$
and $\mu,\nu$ in the compact set $K_L$. Indeed, this is obvious if $V$ has bounded variation and Lipschitz norms, with $\epsilon_{V}(\kappa)=\max\{ \|V\|_{\rm BV}, \|V\|_{\rm L}\} \kappa$. If $V$ is bounded continuous, we let $\eta>0$ and choose ${M=\sqrt{2\eta^{-1} L}}$ so that $\mu([-M,M]^c)+\nu([-M,M]^c)\le \eta$ for $\mu,\nu\in K_L$, and then $V_\eta$ with finite bounded variation and Lipschitz norm so that
$$\sup_{x\in [-M,M]}|V(x)-{ V_\eta(x) }|\le \eta\,.$$
We then check that 
$$ \left|\int V(d\mu-d\nu) \right|\le \|V\|_\infty \eta+ 2\eta + \left|\int V_\eta (d\mu-d\nu) \right|\le  (\|V\|_\infty+ 2)\eta +\max\{ \|V_\eta\|_{\rm BV}, \|V_\eta\|_{\rm L}\} \kappa\,.$$
We finally choose $\eta=\eta(\kappa)$ going to zero slowly enough with $\kappa$ so that the above right hand side goes to zero.
Hence, we can bound the  third term in the right hand side of \eqref{bb2} to find that 
\begin{align*}
&\frac{1}{Z_{N,\T}^{V,{P}}} \int 1_{d(\hat\mu_{L_N({P})}, M_{Q}^V)>\varepsilon} 1_{d(\hat\mu_{L_N({P})},\hat\mu_{L_N({Q})})\le \kappa}{\bf 1}_{ \hat\mu_{L_N({P})},\hat\mu_{L_N({Q})}\in K_L
}
e^{-N\int V(x) d\hat\mu_{L_N({P})}(x)}d\T_N\\
&\qquad \qquad \le e^{N \epsilon_V^L(\kappa) }
\frac{Z_{N,\T}^{V,{Q}}}{Z_{N,\T}^{V,{P}}} \frac{1}{ Z_{N,\T}^{V,{Q}}}
\int 1_{\{d(\hat\mu_{L_N({Q})}, M_{Q}^V)\ge \varepsilon-\kappa\}}  e^{-N\int V(x) d\hat\mu_{L_N({Q})}(x)} d\T_N\,.\end{align*}
Similarly, we find that
\begin{eqnarray*}
Z_{N,\T}^{V,{Q}}&\le & \int e^{-N\int   Vd\hat\mu_{L_N({Q})} }{\bf 1}_{ \hat\mu_{L_N({P})},\hat\mu_{L_N({Q})}\in K_L
}
1_{\{d((\hat\mu_{L_N({P})},\hat\mu_{L_N({Q})})\le \kappa\}}
d\T_N\\
&& +e^{(\|V\|_\infty +c-\sqrt{-\ln|{P}-{Q}|}{ \kappa/2})N} +2 e^{(\|V\|_\infty -cL-C )N}\\
&\le & Z_{N,\T}^{V,{P}}  (e^{N\epsilon_V^L(\kappa)}+e^{({ 2}\|V\|_\infty +c-\sqrt{-\ln|{P}-{Q}|}{ \kappa/2})N}+ 2 e^{(2\|V\|_\infty -cL-C )N})\end{eqnarray*}
where  we used  that the partition function is bounded from below by $e^{-\|V\|_\infty N}$.
Moreover the previous large deviation principle implies if $\kappa\le\varepsilon/2$ that
$$\limsup_{N\rightarrow\infty}\frac{1}{N}\ln \frac{1}{Z_{N,\T}^{V,{Q}}} \int_{\{d(\hat\mu_{L_N({Q})}, M_{Q}^V)\ge \varepsilon-\kappa\}}e^{-N\int V(x) d\hat\mu_{L_N({Q})}(x)} {d\T_N} \le  -\inf_{d(\mu,M_{Q}^V)\ge \varepsilon/2}\{T_{Q}^V\}\,.$$ 
Hence, we find that if $L$ is big enough, ${P}-{Q}$ small enough so that 
$\epsilon_V^L(\kappa)>\max\{ { 2}\|V\|_\infty +c-\sqrt{-\ln|{P}-{Q}|}\kappa, 2\|V\|_\infty -cL-C\}$,  \eqref{bb2} yields 

$$-\inf_{( (M_{Q}^V)^\varepsilon)^c}T_{P}^V\le  2\epsilon_V^L(\kappa)-\inf_{d(\mu,M_{Q}^V)\ge \varepsilon/2}\{T_{Q}^V\}
 $$
 We then conclude that the right hand side is negative for such choices of parameters if  $\kappa$ is small enough and therefore
$\inf_{( (M_{Q}^V)^\varepsilon)^c}T_{P}^V>0$  so that $( (M_{Q}^V)^\varepsilon)^c \subset (M_{P}^V)^c$ which yields the result. 
\end{proof}
\section{$\beta$-ensembles}\label{Coulomb}
\subsection{Large deviation principles for $\beta$-ensembles}
In this section we consider the $\beta$-ensembles and collect already known results about their  large deviation principles. We then relate these large deviation principles with the previous ones thanks to Dumitriu-Edelman tri-diagonal representation, as pioneered in \cite{Spohn1}. 
Coulomb gases  on the real line  are given by  the following $\beta$-ensembles distribution:
\begin{equation}
\label{mesure beta}
\text{d}{\P}^{V,\beta}_{N}(x_1,\cdots,x_N) = \frac{1}{Z^{V,\beta}_{N,C}}\prod_{i<j} |x_i-x_j|^{\beta} e^{-\sum_{i=1}^N (\frac{1}{2}x_i^2 +V(x_i)) }\text{d}x_1\cdots\text{d}x_N.
\end{equation}
$V$ will be a continuous potential. When $V=0$ and $\beta=1$, it is well known \cite[Section 2.5.2]{AGZ} that $d\P^{0,1}_N$ is the law of the eigenvalues of the Gaussian orthogonal ensemble of random matrices  with standard Gaussian entries. Hereafter, we keep the potential to be under the form of a quadratic potential plus a general potential only to have simpler notations later on. In this article, we are however interested in the scaling where  $\beta=\frac{2P}{N}$. The large deviation principles for the empirical measure $\hat\mu_N=\frac{1}{N}\sum_{i=1}^N \delta_{x_i}$ have been derived in \cite{Garcia} and yields the following result.
\begin{theorem}\label{David}\cite{Garcia}
Let $P_N$ be a sequence of positive real numbers converging towards $P>0$. Let 
$W(x)=\frac{1}{2}x^2+V(x)$ be a  continuous function such that for some $P'>P+1$ there exists a finite constant $C_V$ such that for all $x$
\begin{equation}\label{toto} W(x)\ge P'\ln (|x|^2+1)+C_V\end{equation}
Then the law of  $\hat\mu_N$ under $\P^{V,\frac{2P_N}{N}}_N$
satisfies a large deviation principle in the scale $N$ and with good rate function $I_P^V(\mu)=f^{V}_{P}(\mu)-\inf f^{V}_{P}$ where
$$f_P^V(\mu)=\frac{1}{2} \int(W(x)+W(y)-2P\ln |x-y|)d\mu(x)d\mu(y) +\int \ln\frac{d\mu}{dx} d\mu(x)  $$
if $\mu\ll dx$ and $\ln\frac{d\mu}{dx}$ is $\mu$-integrable, whereas $f_P^V$ is infinite otherwise. 
\end{theorem}
In fact, neglecting the singularity of the logarithm, this result would be a direct consequence of Sanov's theorem and Varadhan's lemma. Dealing with  this singularity requires extra-care, a difficulty which was addressed in \cite{Garcia}. Indeed, \cite[Theorem 1.1]{Garcia} can be applied, as was kindly shown to us by David García-Zelada.    For $\frac{1}{2P'}<\alpha <1-\frac{P}{P'}$, we can rewrite 
$$d\P^{V,{\frac{2P_N}{N}}}_N(x_1,\dots,x_N)= \frac{1}{\tilde{Z}^{V,{{P_N}}}_N}e^{-2{{P_N}}N H_N(x_1,\dots,x_N)}d\pi(x_1)\dots d\pi(x_N), $$
where,  if $\gamma(N)=(1-N^{-1})\frac{1}{2P'}+\frac{\alpha-1}{2{{P_N}}} $, we set
$$ H_N (x_1,\dots, x_N) = \frac{1}{N^2}\sum_{1\leq i<j \leq N} \left(\frac{W(x_i)}{2P'} + \frac{W(x_j)}{2P'} - \ln|x_i-x_j|\right)-
\frac{\gamma(N)}{N}\sum_{i=1}^N W(x_i) $$
and $\pi$ is  the probability measure given by
$$ d\pi(x)=\frac{1}{Z}e^{-\alpha W(x)}dx. $$
The sequence $(H_N)_{N\ge 0}$ is { (up to considering $N$ large enough)} uniformly bounded from below by \eqref{toto}. Moreover, letting $\gamma(\infty)=\frac{1}{2P'}+\frac{\alpha-1}{2P}$, we set
 for $\mu\in \mathcal{P}(\R)$,
$$ H(\mu) := \frac{1}{2}\int \left(\frac{W(x)}{2P'}+\frac{W(y)}{2P'} - \ln |x-y|\right)d\mu(x)d\mu(y) -\gamma(\infty) \int W(x)d\mu(x), $$
we find  \cite[Lemma 2.1]{Garcia} that the couple $(\{H_N\}_{N\ge 0},H)$ fulfills the assumptions of \cite[Theorem 1.1]{Garcia}. Thus the law of $\hat\mu_N$ satisfies a large deviation principle at speed $N$ with rate function $I^{V}_{P}=f^V_P - \inf f^V_P$, where

\begin{align*}
f^V_P(\mu) = \left\{  \begin{array}{ll}
				           2P H_{V}(\mu) + \int \ln\frac{d\mu}{d\pi}d\mu &\text{if } \mu \ll \pi\text{ and }\ln\frac{d\mu}{dx}\text{ is }\mu\text{-integrable} \\
				            +\infty								     &\text{otherwise.}
				\end{array}             \right.
\end{align*}
 It is not hard to see that
\begin{lemma} \label{lemreg}For any continuously differentiable  function $W${, any $P'>P+1$} such that \eqref{toto} holds,
\begin{itemize}
\item $\mu\mapsto I_P^V(\mu)$ is strictly convex, 
\item $I_P^V$ achieves its minimal value at a unique probability measure $\mu_P^V(dx)\ll dx$ which satisfies the non-linear equation
\begin{equation}\label{carmu} W(x)- 2P\int\ln |x-y|d\mu_P^V(y)+\ln\frac{d\mu_P^V}{dx} =\lambda_P^V\qquad   a.s \end{equation}
where $\lambda_P^V$ is a finite constant.  Furthermore the support of  $\mu_P^V$  is the whole real line and  the density of $\frac{d\mu_P^{V}}{dx}$ is bounded from above by $C_P (|x|+1)^{2(P-P')}$ where $C_P$ is a constant which is  uniformly bounded on compact subsets of $(0,P'-1)$. 

\item Let $D$ be the distance on $\mathcal P(\R)$ given by
\begin{eqnarray}\label{fourier}D(\mu,\mu')&=&\left(-\int \ln |x-y|d(\mu-\mu')(x)d(\mu-\mu')(y)\right)^{1/2}\nonumber\\
	&=&\left(\int_0^\infty\frac{1}{t}\left| \int e^{it x} d(\mu-\mu')(x)\right|^2 d t \right)^{1/2}\end{eqnarray}
{Then $P\mapsto \mu^V_P$ is locally 1/2-Hölder for the distance $D$: For any $\delta>0$ such that $[P-\delta,P+\delta]\subset (0,P'-1)$, there exists a constant $D>0$ such that for all $P-\delta\leq R \leq P+\delta$, we have
$$D(\mu_P^V,\mu_{R}^V)\le D\sqrt{|P-R|}.$$ }
\end{itemize}
\end{lemma}
We will see later that in fact $P:(0,P'-1)\rightarrow \mu_P^V$ is differentiable, see Lemma \ref{convex}.
Observe that if $f$ is in $L^2$ with derivative in $L^2$, we can set $\|f\|_\frac{1}{2}=(\int_0^\infty t|\hat f_t|^2 dt)^{1/2}$. Then, for any measure $\nu$ with zero mass,
$$\int f(x)d\nu(x)=\int_{-\infty}^\infty \hat f_t \hat\nu_t dt =\int_{-\infty}^\infty \sqrt{t}\hat f_t  \frac{1}{\sqrt{t}}\hat\nu_t dt$$
so that by Cauchy-Schwartz inequality, we get,
\begin{equation}\label{gh}\left| \int f(x)d\nu(x)\right|^2\le \int_{-\infty}^\infty |t \hat f_t |^2 dt \int_{-\infty}^\infty \frac{1}{|t|} |\hat\nu_t|^2 dt= 4\|f\|_{1/2}^2 D(\nu,0)^2\end{equation}
In particular, the last point in the theorem shows that for any $f$ with finite $\|f\|_{1/2}$, $P\rightarrow \int fd\mu_P^V$ is H\"older $1/2$.

\begin{proof}
For $P''>1$, we denote by $\lambda_{P''}$ the probability measure on the real line given by $\lambda_{P''}(dx):=Z_{P''}^{-1} (|x^2|+1)^{-P''/2} dx$ and 
rewrite $f^{V}_{P}$ (up to a constant  $\ln Z_{P''}$) as
$$ f_{P}^{V}(\mu)=\frac{1}{2} \int(\bar W(x)+\bar W(y) -2P\ln |x-y|)d\mu(x)d\mu(y) +\int \ln\frac{d\mu(x)}{d \lambda_{P''}(x) } d\mu(x)  $$
where $\bar W(y):=W(y)- \frac{1}{2} P''\ln (|y|^2+1)$. Because $\lambda_{P''}$ is  a probability measure so that, for every probability measure $\mu$,
$$\int \ln\frac{d\mu}{d\lambda_{ P''}}(x) d\mu(x)\ge 0$$
by Jensen's inequality since $x\mapsto x\ln x$ is convex.

The first point of the lemma is  clear as $\mu\mapsto \int(\bar W(x)+\bar W(y)-2P\ln |x-y|)d\mu(x)d\mu(y)$ is strictly convex \cite[Lemma 2.6.2]{AGZ} whereas the relative entropy $\mu\mapsto \int \ln \frac{d\mu}{d\lambda_{P''}}(y) d\mu(y)$ is well known to be convex.  Since $f^{V}_{P}$ is a good rate function, it achieves its minimal value at a unique  probability measure $\mu_P^V$. Writing that for any measure $\nu$ with mass zero such that $\mu_P^V+\varepsilon \nu$ is a probability measure  for small enough $\varepsilon$, $I_P^V(\mu_P^V+\varepsilon\nu)\ge I_P^V(\mu_P^V)$, we get that \eqref{carmu} holds $\mu_P^V$ almost surely and that the left hand side in \eqref{carmu}  is greater or equal than the right hand side outside of the support of $\mu_P^V$.
Since the left hand side equals $-\infty$ when the density vanishes, we conclude that the support is the whole real line.
We finally show 
 the boundedness of the density. Note that \eqref{carmu} implies that
\begin{equation}\label{for}\frac{d\mu^V_P}{dx}(x)=e^{\lambda_P^V} e^{-W(x)+2P\int \ln |x-y|d\mu_P^V(y)}\end{equation}
We get from \eqref{toto}, and the fact that 
$\ln|x-y|
\le \frac{1}{2}\ln (|x|^2+1)+\frac{1}{2}\ln(|y|^2+1)$ the bound
$$-W(x)+2P\int \ln |x-y|d\mu_P^V(y)\le -(P'-P)\ln (|x|^2+1)+C_V+P\int \ln (|x|^2+1) d\mu_P^V\,.$$
We thus only need to bound $\int \ln (|x|^2+1) d\mu_P^V$ and $\lambda_P^V$ from above. We first notice that $P\mapsto \inf f^{V}_{P}$ is concave since it is the limit of the free energy $-N^{-1}\ln Z_N^{V,\frac{2P}{N}}$. This is enough to guarantee that this quantity is uniformly bounded on compact sets (as it is at any given point). We denote by $C$ such a bound for a fixed compact set.  
As in \cite[Lemma 2.6.2 (b)]{AGZ}, since the relative entropy is non-negative we find that
$$\int (\bar W(x)-P\ln( |x|^2+1))d\mu_P^V (x)\le  f_P^V(\mu_P^V)\le C\,.$$ 
This implies by our hypothesis \eqref{toto} that 
$$(P'-P''-P)\int \ln (|x|^2+1) d\mu_P^V (x)\le  C-C_V$$
 and therefore plugging this estimate in the infimum of $f_P^V$ gives if $P'-P-P''>0$ (which is always possible as we assumed $P'-P>1$)
$$\int W(x) d\mu_P^V (x)\ge C+\frac{C-C_V}{2(P'-P-P'')}$$
Moreover, again because the relative entropy is non-negative,
\begin{eqnarray*}-P\Sigma(\mu^V_P)&:=&-P\int\ln|x-y|d\mu^V_P(x)d\mu_P^V(y)\\
&\le& C-\int \bar W(x)d\mu_P^V(x)\le C-2(P'-P'')\int \ln(|x|^2+1) d\mu_P^V(x)-C_V\end{eqnarray*}
is  uniformly bounded. 
Finally,
from \eqref{carmu} we have  after integration under $\mu_P^V$ 
\begin{equation}\label{zx}\lambda_P^V=\inf f_P^V- P\int \ln |x-y|d\mu_P^V(x)d\mu_P^V(y)\end{equation}
 is thus uniformly bounded from above. This completes the proof of the upper bound of the density:
 $\frac{d\mu_P^V}{dx}$ is bounded by $C_P (|x|+1)^{2(P-P')}$ where $C_P$ is uniformly bounded on compacts so that $P'-P-1\ge \varepsilon>0$ for some fixed $\varepsilon$.
 
 We next study the regularity of the equilibrium measure $\mu^V_P$ in the parameter $P$. {Let $\delta>0$ be such that $[P-\delta,P+\delta]\subset (0,P'-1)$, and let $P-\delta\leq R \leq P+\delta$.}
 If $\Delta \mu=\mu^V_P-\mu^V_R$, since $\mu^V_P$ minimizes $f_P^V$, we have 
\begin{eqnarray*}
0&\ge& f_{P}^V(\mu_P^V)-f_P^V(\mu_R^V) \\
&=&\int W(x) d\Delta \mu(x)  -2P\int \ln|x-y| d\mu_R^V(x)d\Delta \mu(y)-P\int \ln|x-y| d\Delta \mu(x)d\Delta \mu(y)\\
&& +\int\ln \frac{d\mu_P^V}{dx} d\mu_P^V-\int\ln \frac{d\mu_R^V}{dx} d\mu_R^V\\
&=&\int  (2R\int\ln|x-y|d\mu_R^V(y) -\ln\frac{d\mu_R^V}{dx})(x)d\Delta \mu(x) -2P\int \ln|x-y| d\mu_R^V(x)d\Delta \mu(y)\\
&&-P\int \ln|x-y| d\Delta \mu(x)d\Delta \mu(y) 
+\int\ln \frac{d\mu_P^V}{dx} d\mu_P^V-\int\ln \frac{d\mu_R^V}{dx} d\mu_R^V\\
&=&2(R-P) \int \ln|x-y| d\mu_R^V(x)d\Delta \mu(y)- P\int \int \ln|x-y| d\Delta \mu(x)d\Delta \mu(y) +\int\ln \frac{d\mu_P^V}{d\mu_R^V} d\mu_P^V\end{eqnarray*}
where in the second line we used \eqref{carmu} and the fact that $\Delta\mu(1)=0$. 
By using the Fourier transform of the logarithm,  the centering of $\Delta\mu$ and the definition \eqref{fourier} we deduce
\begin{eqnarray}\int\ln \frac{d\mu_P^V}{d\mu_R^V} d\mu_P^V+ P D(\mu_P^V,\mu_R^V)^2&\le& 2(P-R) \int\int \ln|x-y| d\mu_R^V(x)d\Delta \mu(y)\, .\label{po}
\end{eqnarray}
We can assume without loss of generality that  $R<P$. We now show that the integral of the right hand side is bounded independently of $R\in [P-\delta, P]$. We have $\frac{d\mu^V_R}{dx}\leq \frac{C_R}{(1+|x|)^{2(P'-R)}}$, where $R\mapsto C_R$ is bounded on any compact of $(0,P'-1)$, and in particular on $[P-\delta,P+\delta]$. Thus there exists $C>0$ such that $\dfrac{d\mu_R^V}{dx}\leq \frac{C}{(1+|x|)^2}$, and the same bound holds for $\mu^V_P$. Using that for any $x$, $y$ with $x\neq y$ we have $\ln(|x-y|)\leq \ln(1+|x|)+\ln(1+|y|)$ and the previous bound on the density of $\mu^{V}_R$, we conclude that $\int\int \ln|x-y| d\mu_R^V(x)d\Delta \mu(y)$ is uniformly bounded in $R\in [P-\delta,P+\delta]$.
Since $ \int\ln \frac{d\mu_P^V}{d\mu_R^V} d\mu_P^V\ge 0$ by Jensen's inequality
equation \eqref{po} gives the existence of a finite constant $D$  such that
$$D(\mu_P^V,\mu_R^V)\le D\sqrt{|P-R|}\,.$$

\end{proof}

\subsection{Relation with the large deviation principle for Toda matrices with quadratic potential}
When $V=0$, for any $\beta>0$, Dumitriu and Edelman \cite[Theorem 2.12]{dued} have shown that ${\P}^{0,\beta}_{N}$ is the law of the eigenvalues of a $N\times N$ tri-diagonal matrix 
$C_N^\beta$  such that
 $\left((C_N^\beta)_{j,j}\right)_{1\leq j \leq N}$ are independent standard normal variables,   independent from the off diagonal entries $(C_N^\beta)_{j,j+1}=(C_N^\beta)_{j+1,j}$  which are independent and such that $\sqrt{2} C_N^\beta(j,j+1)$ follows a  $\chi_{(N-j)\beta}$ distribution.
    As in the case of  the Toda measure we hereafter identify ${\P}^{0,\beta}_{N}$ with ${\P}^{\beta}_{N}$.
  We are now going to give an alternate large deviation principle for the empirical measure under ${\P}^{2P/N}_{N}$
  based on this representation, this will allow to relate the rate function $I_P=I_P^{0}$ of the Coulomb Gas in terms of the large deviation rate function  $T_s, s\le P$ for Toda matrices.
  \begin{lemma}
\label{approx exp beta ens}
  The law of the empirical measure $\hat\mu_N$ under ${\P}^{2P/N}_{N}$ satisfies a large deviation principle in the scale $N$ and with good rate function 
  \begin{equation}
\label{equation difficile}
 I_P(\mu) = \lim_{\delta \to 0} \liminf_{M\rightarrow\infty} \inf_{\substack{ \nu_{P/M}, \cdots, \nu_P \text{s.t.} \\  \frac{1}{M}\sum_i\nu_{iP/M} \in B_\mu(\delta) } } \left\lbrace \frac{1}{M} \sum_{i=1}^M T_{iP/M}(\nu_{iP/M})\right\rbrace.
\end{equation}
\end{lemma}
Observe for later purpose that we must have $I_{P}=I_{P}^{0}$ where $I_{P}^{V}$ is defined just above Lemma \ref{lemreg}.
\begin{proof} We shall proceed by exponential approximation. We write  $N= k_NM + r_N$, $0\leq r_N \leq M-1$, and consider the matrices
\[
  S_N^M =
  \begin{pmatrix}
    L_{k_N}^1   &   &  &\\
    & \ddots &  &\\
    & &   L_{k_N}^M   &\\
    & & &     0  &
  \end{pmatrix},
\]
with $(L_{k_N}^i)_{1\leq i \leq M}$
a family of independent square matrices  with size $k_N$ distributed according to  $\T_{k_N}^{ \left(P\frac{N-ik_{N}}{N}\right)}$, and a block with null entries of size
 $r_N \times r_N$. We shall prove that  they provide good exponential approximation for  the matrix $C^{\frac{2P}{N}}_N$ following the distribution $ \P^{2P/N}_N$ ,  see \cite[Definition 4.2.14]{DZ}. More precisely, we show that for any positive real number $\delta$ :

 \begin{equation}
\lim_{M\to +\infty} \limsup_{N} \frac{1}{N}\ln  \P( d(\hat{\mu}_{C_N^{\frac{2P}{N}}}, \hat{\mu}_{S^M_N})> \delta) = -\infty\,.
\end{equation}
The lemma is then a direct application of \cite[Theorem  4.2.16 and Exercise 4.2.7]{DZ}.
We first approximate $S^M_N$ by the following matrix 
$$ U^M_N = \left( \begin{array}{rlcl}

\begin{array}{|ccc|}
\hline
~ &  ~ & ~ \\
~ & C_1 & ~ \\
~ & ~  & ~ \\
\hline
\end{array} & \begin{array}{lll}
					~ & ~ &~\\
					~ & ~ & ~\\					
					* & ~  & ~
		                  \end{array} & ~ & ~ \\

\begin{array}{rrr}
 ~&~& * \\					
 ~&~& ~\\
 ~&~&~
       \end{array} &       \begin{array}{c}
						\ddots
      					   \end{array} & \begin{array}{lll}
											~ & ~ &~\\
											~ & ~ & ~\\					
											* & ~  & ~
					 \end{array}& ~   \\

~ & \begin{array}{rrr}
 				~&~& *\\					
 				~&~&~\\
				~&~&~
       		    \end{array} & \begin{array}{|ccc|}
								\hline
								~ &  ~ & ~ \\
								~ &C_M& ~ \\
								~ & ~  & ~ \\
								\hline
							\end{array} & \begin{array}{lll}
												~ & ~ &~\\
												~ & ~ & ~\\					
												* & ~  & ~
					 							\end{array} \\

~ & ~ & \begin{array}{rrr}
 				~&~& *\\					
 				~& ~&~\\
				~& ~&~
       		    \end{array} &\begin{array}{|ccc|}
								\hline
								~ &  ~ & ~ \\
								~ & R^M_N & ~ \\
								~ & ~  & ~ \\
								\hline
							\end{array}

\end{array}\right), $$
where the symbols 
 $*$ denote entries following the law of a matrix distributed according to $\P^{2P/N}_N$ :
$$U_N(ik_N,ik_N+1)=U_N(ik_N+1,ik_N)\sim \frac{1}{\sqrt{2}}\chi_{2P\frac{N-ik_N}{N}} ,\ 1\leq i \leq M\ ;$$
$R^M_N$ has same distribution as the $r_N\times r_N$-bottom-right corner of a $\P^{2P/N}_N$- distributed matrix. $C_i$ has the same coefficients as $L^i_{k_N}$ except for the top-right and bottom-left corner entries which are  put to zero : 
$$C_i = 
\begin{pmatrix}
g_{(i-1)k_N+1} &  \ddots  & ~        &  0      \\
\ddots     &  \ddots  & \frac{1}{\sqrt{2}}c_j^i   & ~        \\
      ~     &  \frac{1}{\sqrt{2}}c_j^i    & \ddots  & \ddots \\
     0     &      ~     & \ddots  & g_{ik_N}
\end{pmatrix}.
 $$
The $(c_j^i)_{1\leq j \leq k_N-1}$ are distributed according to  $\chi_{2P\frac{N-ik_N}{N}}$.\\
For $1\le i \le M$ and $1\le j \le k_N-1$, let $b^i_j = \sqrt{(c^i_j)^2 + \chi_{i,j}^2}$,
where $(\chi_{i,j})_{1\le i \le M, 1\le j \le k_N}$ is an independent family of $\chi$ variables with parameter $2P\frac{k_N-j}{N}$, independent from $U^N_M$.\\
We set, for $1\le i \le M$, $B_{i}$ to be the matrix
$$B_i =  
\begin{pmatrix}
g_{(i-1)k_N+1} &  \ddots  & ~        &  0      \\
\ddots     &  \ddots  & \frac{1}{\sqrt{2}}b_j^i   & ~        \\
      ~     &  \frac{1}{\sqrt{2}}b_j^i    & \ddots  & \ddots \\
     0     &      ~     & \ddots  & g_{ik_N}
\end{pmatrix}\,. $$
The matrix
$$C^{2P/N}_N = 
\left( \begin{array}{rlcl}

\begin{array}{|ccc|}
\hline
~ &  ~ & ~ \\
~ & B_1 & ~ \\
~ & ~  & ~ \\
\hline
\end{array} & \begin{array}{lll}
					~ & ~ &~\\
					~ & ~ & ~\\					
					* & ~  & ~
		                  \end{array} & ~ & ~ \\

\begin{array}{rrr}
 ~&~& * \\					
 ~&~& ~\\
 ~&~&~
       \end{array} &       \begin{array}{c}
						\ddots
      					   \end{array} & \begin{array}{lll}
											~ & ~ &~\\
											~ & ~ & ~\\					
											* & ~  & ~
					 \end{array}& ~   \\

~ & \begin{array}{rrr}
 				~&~& *\\					
 				~&~&~\\
				~&~&~
       		    \end{array} & \begin{array}{|ccc|}
								\hline
								~ &  ~ & ~ \\
								~ &B_M& ~ \\
								~ & ~  & ~ \\
								\hline
							\end{array} & \begin{array}{lll}
												~ & ~ &~\\
												~ & ~ & ~\\					
												* & ~  & ~
					 							\end{array} \\

~ & ~ & \begin{array}{rrr}
 				~&~& *\\					
 				~& ~&~\\
				~& ~&~
       		    \end{array} &\begin{array}{|ccc|}
								\hline
								~ &  ~ & ~ \\
								~ & R^M_N & ~ \\
								~ & ~  & ~ \\
								\hline
							\end{array}
\end{array}\right)$$
is distributed according to $\P^{2P/N}_N$, where the symbols $*$ denote the same coefficients as those of $U^M_N$. Because the rank of 
$S^M_N-U^M_N$ is bounded by $ 2M+r_N \leq 3M$, by (\ref{super inegalite}) we have
\begin{equation}
d(\hat{\mu}_{U^M_N}, \hat{\mu}_{S^M_N})\leq \frac{3M}{N} = \frac{3}{k_N}.
\end{equation}
Let $\delta$ be a positive real number. Then for $N$ large enough so that $k_N$ verifies $\frac{3}{k_N}\leq \delta/2$,
\begin{align*}
\P\left( d(\hat{\mu}_{C_N^{2P/N}},\hat{\mu}_{S^M_N} ) >\delta \right) &\leq \P\left(d(\hat{\mu}_{C_N^{2P/N}},\hat{\mu}_{U^M_N} ) + d(\hat{\mu}_{U^M_N},\hat{\mu}_{S^M_N} )>\delta \right) \\
									      & \leq \P\left( d(\hat{\mu}_{C_N^{2P/N}},\hat{\mu}_{U^M_N} ) >\delta/2 \right).
\end{align*}
Moreover (\ref{super inegalite}) yields
\begin{equation}
d(\hat{\mu}_{U^M_N},\hat{\mu}_{C_N^{2P/N}}) \leq \frac{2}{N}\sum_{i=1}^{N}|Y_i|,
\end{equation}
where { $Y_i$}  is the $i$th  coefficient above or below the $(i,i)$ the coefficient of $C_N^{2P/N}-U^M_N$. Applying the inequality $\sqrt{a + b}\leq \sqrt{a}+\sqrt{b}$ for $a,b\geq 0$ and $a= c_j^i$ and $b=\chi_{i,j}$, we deduce
\begin{equation}
d(\hat{\mu}_{U^M_N},\hat{\mu}_{C_N^{2P/N}}) \leq 
\frac{\sqrt{2}}{k_NM}\sum_{i=1}^{k_NM} \chi^i_{2P/M},
\end{equation}
where the last sum denotes the sum of iid variables with law  $\chi_{2P/M}$ { (and we used that there exists a coupling between a $\chi_{2P \frac{k_N-j}{N}}$ and a $\chi_{2P/M}$ variable such that the first is always bounded above by the second).}

Thus for all  $\delta>0$, for any integer numbers  $N$ such that $\frac{3}{k_N}\leq \delta/2$ (i.e for $N$ larger than some $N_0$ depending on $M$) and for any non-negative function  $A\ :\ M\mapsto A(M)$
\begin{align*}
\P\left( d(\hat{\mu}_{S^M_N},\hat{\mu}_{C_N^{2P/N}}) > \delta \right) &\leq \P\left( \sum_{i=1}^{k_NM} \chi_{2P/M} > \frac{k_NM\delta}{2\sqrt{2}} \right) \\
																&\leq e^{-A(M)k_NM\delta/(2\sqrt{2})} \mathbb{E}\left[ e^{A(M)\chi_{2P/M}} \right]^{k_NM}.
\end{align*}
It is not hard to see that with 
$A(M)=\sqrt{\ln (M)}$, there exists a finite constant $K$ such that 
\begin{equation}\label{mo} \sup_{M\ge 0}\mathbb{E}\int e^{A(M)x} d\chi_{1/M}(x)  \le K\end{equation}
 insuring that 
$$
\frac{1}{N}\ln  \P ( d(\hat{\mu}_{C_N^{2P/N}},\hat{\mu}_{S^M_N}>\delta) \leq -A(M)\frac{\delta}{2\sqrt{2}} + K,
$$
which yields the result. 

\end{proof}

We shall use the previous lemma to study the case with a non trivial potential. Indeed, as a direct consequence of Lemma \ref{approx exp beta ens} and Varadhan's lemma, we deduce the following Theorem. 
\begin{theorem}\label{theobounded}
For any continuous function $V$ such that
\begin{equation}\label{az}
\limsup_{|x|\rightarrow \infty}\frac{|V(x)|}{x^{2}}=0,\end{equation}
the law of the empirical measure $\hat\mu_N$ under ${\P}^{V,2P/N}_{N}$ satisfies a large deviation principle in the scale $N$ and with good rate function $I_P^{V}(\mu)=f_P^V(\mu) -\inf f_P^{V} $ where
  \begin{equation}
\label{equation difficile}
 f_P^{V}(\mu) =\lim_{\delta \to 0} \liminf_M \inf_{\substack{ \nu_{P/M}, \cdots, \nu_P \text{s.t.} \\  \frac{1}{M}\sum_i\nu_{iP/M} \in B_\mu(\delta) } } \left\lbrace \frac{1}{M} \sum_{i=1}^M (T_{iP/M}(\nu_{iP/M})+\int V d\nu_{iP/M})\right\rbrace.
\end{equation}
\end{theorem}
\begin{remark} Varadhan's lemma gives the result for bounded continuous function $V$. However, we can approximate $V$ by $V(x)(1+\epsilon x^{2})^{-1}$ with overwhelming probability thanks to Lemma \ref{exptight}, which  allows to conclude for any potential $V$ satisfying \eqref{az}

\end{remark}
We shall use this relation to give a better description of the rate function $T_P$. In fact we first consider the free energy
$$F_{\T}^{V,P}=\lim_{N\rightarrow
\infty}\frac{1}{N}\ln Z_{N,\T}^{V,P}, \, F_C^{V,P}= \lim_{N\rightarrow\infty}\frac{1}{N}\ln Z_{N,C}^{V,P}=-\inf f^{V}_{P}.$$

\begin{lemma} \label{convex}For any continuous function $V$ satisfying \eqref{az},
	\begin{itemize}
		\item $P\mapsto F^{V,P}_C=-\inf f_P^{V}$ is continuously differentiable on $(0,+\infty)$. Moreover, for any $P>0$
		$$F^{V,P}_\T=\partial_P(PF^{V,P}_C)$$
		\item  {For any bounded continuous  function $f$, the map $P\in (0,+\infty)\mapsto P\mu^V_P(f)$ is continuously differentiable.} Moreover,  there exists a unique minimizer  $\nu^{V}_{P}$ of $\mu\mapsto T_P(\mu)+\int Vd\mu(x)$, {which satisfies, for any bounded continuous function $f$, 
			$$ \nu^V_P(f) = \partial_P(P\mu^V_P(f)).$$
		}
		Therefore, we have
		\begin{equation}\label{egalite mesures minimisantes}\nu^V_P=\partial_P (P \mu^V_P)\,.\end{equation}
		\item For any probability measure $\mu$,
		\begin{equation}
			T_P(\mu)=-\inf_{V\in \mathcal{C}^0_b}\left\{ \int_\R Vd\mu + F_\T^{V,P}  \right\}.
		\end{equation}
	\end{itemize}
\end{lemma}
\begin{proof}
First notice that, for any probability measure $\mu$, Lemma \ref{approx exp beta ens} implies that
 
\begin{align}
f^{V}_{P}(\mu)=I_{P}(\mu)+\int_\R V \text{d}\mu &\ge 
						\liminf_M \frac{1}{M} \sum_{i=1}^M \inf_\nu \left\{ T_{iP/M}(\nu) + \int_\R V d\nu \right\}\nonumber\\
								    &= \int_0^1 \inf_\nu \left\{ T_{sP}(\nu) + \int_\R V \text{d}\nu \right\}\text{d}s= { -\int_0^1 F_\T^{V,sP} ds}\,.\label{ineq1}
\end{align}
{In the equality between the $\liminf$ and the integral, we used the fact that $s\in (0,1)\mapsto F_\T^{V,sP}$ is convex and therefore continuous.} We claim that this lower bound is achieved. 
For  $s\in [0,1]$, let  $\nu_{sP}^*$ be  a minimizer of  $\mu \mapsto T_{sP}(\mu) + \int V\text{d}\mu$. By Corollary \ref{continuite minimiseur}, we can choose 
 $\nu_{sP}^*$ such that  $s\mapsto \nu_{sP}^*$ is continuous.  Hence, $\mu_P^*:=\int_0^1 \nu_{sP}^* \text{d}s$ makes sense and is a probability measure on $\R$. We claim it minimizes $f_P^{V}$. Indeed, by Lemma \ref{approx exp beta ens}, we have 
 \begin{align}
f_P^{V}(\mu_P^*)&= \lim_{\delta \to 0} \liminf_M \inf_{\frac{1}{M}\sum_{i=1}^M \nu_{iP/M}\in B_{\mu_{P}^*}(\delta)}\left\{ \frac{1}{M}\sum_{i=1}^M T_{iP/M}(\nu_{iP/M}) + \int_\R V d\nu_{iP/M} \right\} \nonumber\\
								    &\le\liminf_M \frac{1}{M} \sum_{i=1}^M \left\{ T_{iP/M}(\nu^*_{iP/M}) + \int_\R Vd\nu^*_{iP/M} \right\} \nonumber \\
								    &=\liminf_M \frac{1}{M} \sum_{i=1}^M \inf_\nu \left\{ T_{iP/M}(\nu) + \int_\R V d\nu \right\}\nonumber\\
								    &= \int_0^1 \inf_\nu \left\{ T_{sP}(\nu) + \int_\R V \text{d}\nu \right\}\text{d}s{{= -\int_0^1 F_\T^{V,Ps} ds.}}
								    \nonumber
\end{align}
With \eqref{ineq1}, we deduce that the above inequality is an equality and that  $f^{V}_{P}$ achieves its minimal value at  $\mu^{*}_{P}$.
 By Lemma \ref{lemreg}, this minimizer is unique and therefore $\mu_P^*=\mu_P^V$ for any choices of paths $\nu_.^*$ and any positive real number $P$.
Hence, we find that 
$$-F_C^{V,P}=\inf f_P^V=I_P(\mu_P^V)+\int_\R V \text{d}\mu_P^V =-\int_0^1 F_\T^{V,Ps} ds\,.$$
By a change of variable we deduce
$$P F_C^{V,P}=\int_0^P F_\T^{V,s} ds\,.$$
Since $s\mapsto F_\T^{V,s}$ is convex, it is continuous. This shows that $P\mapsto PF_C^{V,P}$ is continuously differentiable, and that
for all $P>0$,
$$F_\T^{V,P} =\partial_P (P F_C^{V,P})\,.$$
Moreover, we have seen that for any choice of continuous minimizing path $\nu_\cdot^*$ of $\mu\mapsto T_{\cdot}(\mu)+\int Vd\mu$ and any positive real number $P$,
$$\mu_{P}^{V}=\int_{0}^{1}\nu^{*}_{sP} ds=\frac{1}{P}\int_{0}^{P}\nu^{*}_{s}ds\,.$$
{Integrating the last equality against $f$ bounded continuous we have
$$\mu_{P}^{V}(f) = \frac{1}{P}\int_{0}^{P}\nu^{*}_{s}(f)ds\,.$$
By continuity of $s\mapsto \nu^{*}_{s}(f)$, we deduce that $P\mapsto \mu_{P}^{V}(f)$ is continuously differentiable and that 
$$\nu_P^*(f) = \partial_P(P\mu^V_P(f)).$$
But   Corollary \ref{lemreg} implies that any probability which minimizes $T^V_P$ can be seen as the endpoint of a continuous path $s\in (0,P]\mapsto \nu^*_s$ where each $\nu^*_s$ minimizes $T^V_s$. By the latter, such a measure is then equal to $\partial_P(P\mu^V_P(f))$, showing the uniqueness of the minimizer $\nu^V_P$ of $T^V_P$ and the equality 
$$\nu^V_P = \partial_P(P\mu^V_P).$$}\\
The last point of the Lemma is a direct consequence of \cite[Theorem 4.5.10]{DZ} since $T_P^V$ is convex for all bounded continuous function $V$. 

\end{proof}

By Lemma \ref{lemreg}, $\nu^V_P$  is a probability measure which satisfies almost surely
$${d\nu^V_P}(x)=(C_P^V+2P\int \ln|x-y|d\nu_P^V(y) )d\mu^V_P(x)$$
with $C_P^V$ a constant such that
$$C_P^V+2P \int  \ln|x-y|d\nu_P^V(y) d\mu^V_P(x)=1$$
Furthermore we must have $C_P^V+2P\int \ln|x-y|d\nu_P^V(y)\ge 0$ for all $x$.

\section{Large deviations for Toda Gibbs measure with general potentials}\label{general}
We now consider the measures $\T_{N}^{V,P}$ given by (\ref{Gibbs}), with {potential given by} $W:x\in \R \mapsto a x^{2k}+U(x)
$, $k\geq 2$, with $U(x)/x^{2k}$ going to zero at infinity. We show that under these laws, the law of the empirical measures $(\hat{\mu}_{L_N})_{N\geq 1}$ still fulfills a large deviation principle, by extending the subadditivity argument previously used. We then identify the rate function as before. By Varadhan's Lemma, it is enough to consider the case where $U(x)=\frac{1}{2} x^{2}$  {(we detail this in Section \ref{proof: finale})}. We hereafter continue to use the notation \eqref{Gibbs} with now $V(x)= a x^{2k}$.

\subsection{Exponential tightness}
\label{exp tight}
In this section we prove that if $W(x)= a x^{2k} {+\frac{1}{2}x^2}${, \textit{i.e} $V(x)= a x^{2k} $}
with  $k\geq 2$ and  $a>0$, then the law of the empirical measure of the eigenvalues is exponentially tight under  $\T_{N}^{V,P}$. More precisely, we let 
$\mathcal K_L = \{ \mu \in \mathcal{P}(\R)\ |\ \int V(x) d\mu(x) \leq L \}$ which is a compact of $\mathcal{P}(\R)$. Then we shall prove
\begin{lemma}\label{exptg} There exists a finite constant $c_W$ such that
$$
\T_{N}^{V,P}({\hat{\mu}_N\in\mathcal{K}_L^{c}}) \le e^{-(L-c_W)N}.
$$
\end{lemma}
\begin{proof}
 We first  bound from below the free energy  by Jensen's inequality

\begin{equation}\label{blk}
Z_{N,\T}^{V,P} = \int_{\R^{2N}} e^{-N\int_\R V d\hat{\mu}_N}d\T^{P}_N \geq \exp\{ -N
\int_{\R^{2N}} \int_{\R}V d\hat{\mu}_N d\T^{P}_N \}\ge \exp\{{-c_{V} N}\}\,.\end{equation}
From here we deduce exponential tightness for $(\hat{\mu}_N)_N$ under $\T_N^{V,P}$ : for $L>0$, 
\begin{align}
\T_{N}^{V,P}\left( \int_\R V d\hat{\mu}_N \geq L \right) &= \frac{1}{Z_{N,\T}^{V,P}}\int_{\R^{2N}} \mathbf{1}_{ \left\{ \int_\R Vd\hat{\mu}_N \geq L \right\} } e^{-N\int_\R V d\hat{\mu}_N}d\T_{N}^{P} \nonumber\\
													 &\leq e^{N(c_{V}-L)}\label{b1}.
\end{align}
\end{proof}
For later purpose we prove the following result showing that the off diagonal entries $b_{i}=e^{-r_{i}/2},{1\le i\le N}$ of the Lax matrix $L_{N}$ do not become too small :
\begin{lemma}\label{exptlog} For any $P>0$
$$
\limsup_L \limsup_N \frac{1}{N}\ln \T_{N}^{V,P}\left(\frac{1}{N}\sum_{{i=1}}^{N}\ln b_{i }\le -L\right) = -\infty.
$$
\end{lemma}
\begin{proof}
Since $V$ is bounded from below and we have bounded from below the partition function \eqref{blk}, it enough to prove this estimate when $V=0$. But, in this case the entries are independent and so we only need to prove it for independent chi distributed variables. But
then, for any $0<\delta<P$, with $\mathbb Z_{N,\T}^{P}=\mathbb Z_{N,\T}^{0,P}$ the partition function in \eqref{pp}, we find
$$\T_N^{P}\left( \frac{1}{N}\sum_{{i=1}}^{N}\ln b_{i }\le -L\right)\le e^{-\delta LN }\frac{\mathbb Z_{N,\T}^{P-\delta/2}}{\mathbb Z_{N,\T}^{P}}= e^{-\delta LN }\left(\frac{\Gamma(P-\delta/2)}{2^{{\delta/2}}\Gamma(P)}\right)^{N}
$$
from which the result follows by taking for instance $\delta=P/2$. 

\end{proof}

\subsection{Weak LDP}
\label{weak ldp unbounded}

In this section, we prove that $\hat{\mu}_{L_N}$  satisfies a weak large deviation principle, namely Lemma \ref{wldp}.  In this more general setup, we follow again a  subadditivity argument, which is however more sophisticated since the entries of $L_N$ are not independent anymore. We will restrict ourselves to the case where  $V(x)=a x^{2k}$, $a>0$, the case of a more general potential with the same asymptotic behavior being again a consequence of Varadhan's Lemma. We first show that the large deviation principles
is the same if we remove the entries (equal to $b_{N}$) in the corners $(N,1)$ and $(1,N)$ in the Toda matrix. Namely, let $\tilde L_N$ be the tridiagonal matrix  with entries equal to those of $L_N$ except for the entries $(1,N)$ and $(N,1)$ which vanish  and consider the following modification of $\T^{V,P}_N$ given by 
\begin{equation}
	\label{def T tilde}
	d\tilde \T^{V,P}_N=\frac{1}{\tilde Z_{N}^{V,P} }e^{-\tr V(\tilde L_N)} d\T^P_N\,.
\end{equation}

\begin{lemma}\label{expap} For any  probability measure $\mu$, we have
$$\lim_{\delta\rightarrow 0}
\liminf_{N\rightarrow\infty} \frac{1}{N}\ln \int  1_{d(\hat\mu_{L_N},\mu)<\delta} e^{-\tr V(L_N)} d\T^P_N=\lim_{\delta\rightarrow 0} \liminf_{N\rightarrow\infty} \frac{1}{N}\ln \int  1_{d(\hat\mu_{\tilde L_N},\mu)<\delta} e^{-\tr V(\tilde L_N)} d\T^P_N$$
 Moreover, 
$$
\liminf_{N\rightarrow\infty} \frac{1}{N}\ln \int  e^{-\tr V(L_N)} d\T^P_N=\liminf_{N\rightarrow\infty} \frac{1}{N}\ln \int  e^{-\tr V(\tilde L_N)} d\T^P_N\,.$$
The same results hold if we replace all the liminf by limsup. 
\end{lemma}
\begin{proof}To simplify the notations we take $a=1$ in the proof. First notice that  $V(L_N)-V(\tilde L_N)$ is an homogeneous polynomial of degree $2k$ in $L_{N}$ and $\Delta L_N=L_{N}-\tilde L_{N}$, with degree at least one in the latter. Observe that $\Delta L_N$ only depends on $b_N$.
Therefore, there exists a finite constant $C_k$ such that on $B_N^{{K,M}}:=\{b_N\le K\}\cap \{ \frac{1}{N}{\tr (L_N^{2k})}\le M\}$ ( or $\tilde B_N^{M,K}:=\{b_N\le K\}\cap \{ \frac{1}{N}{\tr (\tilde L_N^{2k})}\le M\}$),
H\"older's inequality implies 
\begin{align*}
\left| \frac{1}{N}\tr \left(V(L_N)-V(\tilde L_N)\right)\right|&\le {C_k}  \sum_{l=1}^{2k}\left(\frac{1}{N}\tr\left((\Delta L_N)^{2k}\right)\right)^{l/2k}\left(\frac{1}{N}\tr\left(L_N^{2k}\right)\right)^{\frac{2k-l}{2k}}  \\ 
&\le C(M,K) N^{-\frac{1}{2k}}
\end{align*}
where 
{$C(M,K)$ is a finite constant} depending only on  $M,K,k$. Note  that in the above right hand side $\tr (L_N^{2k})$ can be replaced by $\tr(\tilde L_N^{2k})$ as they play a symmetric role.
Moreover, by \eqref{bound12}, $d(\hat\mu_{L_N},\hat\mu_{\tilde L_N})\le 2/N$ since $\Delta L_N$ has { rank at most two}. 
We fix a probability measure $\mu$ and first prove that
\begin{equation}\label{bi}
\liminf_{N\rightarrow\infty} \frac{1}{N}\ln \int  1_{d(\hat\mu_{L_N},\mu)<\delta}  e^{-\tr V(L_N)} d\T^P_N\ge
 \liminf_{N\rightarrow\infty} \frac{1}{N}\ln { \int 1_{d(\hat\mu_{L_N},\mu)<\delta} e^{-\tr V(\tilde L_N)} d\T^P_N}\,.\end{equation}
We can assume without loss of generality  that the right hand side does not equal $-\infty$.
Then, we have by the previous remark 
\begin{eqnarray*}
\int  1_{d(\hat\mu_{L_N},\mu)<\delta} e^{-\tr V(L_N)} d\T^P_N &\ge& e^{-C(M,K)N^{\frac{2k-1}{2k}} } \int  1_{\tilde B_N^{M,K}\cap \{d(\hat\mu_{\tilde L_N},\mu)<\delta-\frac{2}{N}\}} e^{-\tr V(\tilde L_N)} d\T^P_N\\
&\ge& C'e^{-C(M,K)N^{\frac{2k-1}{2k}} } \int  1_{\{\tr V(\tilde L_N)\le NM\}\cap \{d(\hat\mu_{\tilde L_N},\mu)<\delta-\frac{2}{N}\}} e^{-\tr V(\tilde L_N)} d\T^P_N\\
&\ge& C'e^{-C(M,K)N^{\frac{2k-1}{2k}} }\left\lbrace  \int  1_{ \{d(\hat\mu_{\tilde L_N},\mu)<\delta-\frac{2}{N}\}} e^{-\tr V(\tilde L_N)} d\T^P_N-e^{-NM}\right\rbrace\\
\end{eqnarray*}
where in the second line we integrated over $b_N\le K$ and in the last line we used that 
$$\int  1_{\{\tr V(\tilde L_N)\ge NM\}} e^{-\tr V(\tilde L_N)} d\T^P_N\le e^{-NM}\,.$$
We next choose $M$ so that this term is smaller than the first term (which we assumed bounded below by  $e^{-N C}$ for some finite $C$).
We deduce that \eqref{bi} holds.
To prove the converse inequality, we notice that there exists one $b_i$ bounded by $K$ with probability greater than $1-e^{-a(K) N}$ under $\T^P_N$, with $a(K)=-\ln P(b\ge K)>0$ which goes to $+$ infinity when $K$ does. By symmetry with respect to the order of  the indices, we may assume it is $b_N$. Therefore, because  $V$ is bounded below by some finite constant $C$, setting $a'(K)=a(K)-C$,  and using  Lemma \ref{exp tight},  we find 
\begin{eqnarray*}
&&\int  1_{d(\hat\mu_{L_N},\mu)<\delta} e^{-\tr V(L_N)} d\T^P_N \le e^{-Na'(K)}+ N \int  1_{\{b_N\le K\}
\cap \{d(\hat\mu_{L_N},\mu)<\delta \}} e^{-\tr V( L_N)} d\T^P_N\\
&&\qquad \le e^{-Na'(K)}+ Ne^{-N(M-c_V)}+ N e^{C(M,K)N^{\frac{2k-1}{2k}} } \int  1_{ B_N^{M,K}\cap \{d(\hat\mu_{\tilde L_N},\mu)<\delta+\frac{2}{N}\}} e^{-\tr V(\tilde L_N)} d\T^P_N\\
&&\qquad \le e^{-Na'(K)}+ Ne^{-N(M - c_V) }+Ne^{C(M,K)N^{\frac{2k-1}{2k}}}  \int  1_{ \{d(\hat\mu_{\tilde L_N},\mu)<\delta+\frac{2}{N}\}} e^{-\tr V(\tilde L_N)} d\T^P_N\\
\end{eqnarray*}
which gives the converse bound, letting $N$ going to infinity, provided $K$ and $M$ are large enough. The same arguments also hold when there is no indicator function, giving the same estimates for the free energy.
\end{proof}

\begin{lemma} \label{wldpT} Let $V(x)=a x^{2k}$ and $P>0$. 
For any $\mu$ in $\mathcal P(\R)$,
there exists a limit
\begin{equation}
\label{PGD faible}
\lim_{\delta \to 0} \liminf_N \frac{1}{N}\ln  \T^{V,P}_N\left( \hat{\mu}_{L_N} \in B_\mu(\delta)\right) = \lim_{\delta \to 0} \limsup_N \frac{1}{N}\ln  \T^{V,P}_N\left( \hat{\mu}_{L_N} \in B_\mu(\delta)\right).
\end{equation}
We denote this limit by $-T^V_P(\mu)$. {Then, $\mu \mapsto T^V_P(\mu)$ is convex.}
\end{lemma}
\begin{proof}
We use the notations of Lemma \ref{wldp}. Let $q\geq 1$ be fixed. For $N\geq 1$ we write $N=k_Nq+r_N$, $0\leq r_N \leq q-1$, and define $L^q_N$  by removing the off diagonal entries 
$b_{\ell q}= {L_N(\ell q,\ell q+1)}, L_N(\ell q +1,\ell q), 1\le \ell\le k_N$,  as well as the entries ${L_N(1,N)},{L_N(N,1)}$, from $ L_N$. 
We set $R^q_N=L_N-L_N^q$.
Let $Z_N^V=Z_{N,\T}^{V,P}$ denote in short the partition function for the Toda Gibbs measure with potential $V$ and set
$$Z_{N,q}^V = \E_{\T^{P}_N}\bigg[ e^{-\tr{V(L^q_N)}} \bigg] = \int e^{-\tr{V(L^q_N)}} d\T^{P}_N.$$

We first show that there is some constant $C_k$ (independent of $N$) such that for all $N\geq 1$,
\begin{equation}\label{bound1}
\frac{1}{N}\ln\frac{Z_{N,q}^V}{Z_{N}^V} \geq {-}\frac{C_k}{q^{1/2k}}.  
\end{equation}
By Jensen's inequality we have
\begin{equation}\label{bn0}\frac{1}{N}\ln\frac{Z_{N,q}^V}{Z_N^V} =\frac{1}{N} \ln\E_{\T^{V,P}_N}\bigg[ e^{\tr(V( L_N )- V(L^q_N))}\bigg] \geq \frac{1}{N}\E_{\T^{V,P}_N}\bigg[\tr(V( L_N)-V(L^q_N))\bigg]\,.\end{equation}
{As in the proof of  Lemma \ref{exptg}, we} bound the right hand side by first noticing that $V(L_N)-V(L^q_N)$ is an homogeneous polynomial of degree $2k$ in $L_{N}$ and $ L_{N}-L^{q}_{N}$, with degree at least one in the latter.
Therefore, H\"older's inequality implies that there exists a finite constant $C$ depending only on $k$ such that 
$$\left| \frac{1}{N}\E_{\T^{V,P}_N}\bigg[\tr(V( L_N)-V(L^q_N))\bigg]\right| \le C \sum_{l=1}^{2k}\E_{\T^{V,P}_N}\bigg[
\frac{1}{N} \tr\left(( L_N-L^q_N)^{2k} \right)\bigg]^{l/2k}  \E_{\T^{V,P}_N}\bigg[ \frac{1}{N} \tr ( L_N^{2k})  \bigg]^\frac{2k-l}{2k}$$
Now, $R_N^q= L_N-L^q_N$ has non zero entries only at the sites $(i,i+1)$ and $(i+1,i)$, $i\in J=\{ \ell q, 1\le \ell \le k_N\}$, as well as $(N,1)$ and $(1,N)$.  We can assume without loss of generality that $q> 2k$ so that $\tr (R_N^q)^{2k}$ simply depends on the $2k$th power of the its non-vanishing entries. 
Thus, there exists a finite constant $C_k$ which only depends on $k$ such that
$$\tr\left((R_N^q)^{2k}\right)\le C_k \sum_{i\in J}  {L_N(i,i+1)}^{2k}+C_{k} {L_N(N,1)}^{2k}
\,.$$
{Next notice that 
$$L_N(i,i+1)^2 \leq L_N(i,i)^2+L_N(i,i+1)^2+L_N(i,i-1)^2=L_N^2(i,i) . $$
Moreover, diagonalizing  $L_{N}=\sum \lambda_{j} v_{j}v_{j}^{T}$, we find by H\"older's inequality  (since $\sum v_{j}(i)^{2}=1$ for all $i\in \{1,\ldots,N\}$) that
$$ L_N^2(i,i)^k=\left(\sum \lambda_{j}^{2}  v_{j}(i)^{2}\right)^{k}\le  \sum \lambda_{j}^{2k}  v_{j}(i)^{2}= L_N^{2k}(i,i).$$}
Thus,
$$L_N(i,i+1)^{2k}\le L_N^2(i,i)^k\le  L_N^{2k}(i,i)\,.$$
Because $L_N$ has  periodic boundary conditions, the distribution of the entries of $L_N$ are invariant under the shift $\theta:i\rightarrow i+1$, so that  under  $\T^{V,P}_N$, ${L_N(i,i+1)}$ has the same law than ${L_N(i+1,i+2)}${, and $L_N(i,i)$ has the same law than $L_N(i+1,i+1)$}.  As a consequence, we have
$$\E_{\T^{V,P}_N}\bigg[ 
\frac{1}{N} \tr\left((L_N-L^q_N)^{2k} \right)\bigg]\le \frac{1}{N} C_k\sum_{i\in J}  \E_{\T^{V,P}_N}\bigg[ 
{L_N^{2k}(i,i)}\bigg]= C_k\frac{k_N}{N} \E_{\T^{V,P}_N}\bigg[ \frac{1}{N}\tr
{(L_N^{2k})}\bigg] \,.$$ 
But \eqref{b1} implies that $\E_{\T^{V,P}_N}\bigg[ \frac{1}{N}\tr (L_N^{2k})
\bigg]$ is bounded by some finite constant independent of $N$. We therefore deduce \eqref{bound1} from \eqref{bn0}.

We next prove the subadditivity property.
Let $\delta >0$ and $L>0$ be given. Let $\mathcal K_L=\{\hat\mu_{L_N}(V)\le L\}$.  
As in equation (\ref{inegalite R}), we have for $q$ big enough,
\begin{align}
\T^{V,P}_N\left(\{\hat{\mu}_{ L_N}\in B_\mu(\delta)\} \cap \mathcal K_{L} \right) 
																 &\geq \frac{Z^V_{N,q}}{Z^V_N}\frac{1}{Z^V_{N,q}} \int_{\mathcal K_L\cap K_{A}} \mathbf{1}_{\hat{\mu}_{L_N^q}\in B_\mu(\delta-4/q)} e^{-\tr(V( L_N))}d{\T}^{P}_N, \label{as}
\end{align}
where we set $K_A = K_{A,N} =\cap_{i\in J} \{\ b_i^{2k}\leq A\}\cap\{b_N^{{2k}}\le A\}$.  
As before, noticing that $V(L_N)-V(L_N^q)$ is a polynomial in $L_N^q$ and $L_N-L_N^q$, we find a finite constant $C$ such that, on $\mathcal K_L\cap K_{A}$, for $N$ large enough,

$$\frac{1}{N} \left|\tr(V(L_N)-V(L^q_N))\right| \le C
\left(\frac{k_N}{N} C_k A \right)^{1/2k} L^\frac{2k-1}{2k}\,.$$ Therefore  if we set $\mathcal K_{L}^{q}=\{\hat\mu_{L_N^q}(V)\le L\}$, we deduce that $K_A\cap \mathcal K_L$ contains $K_A\cap\mathcal K_{L-\varepsilon(q)}^q$ for some $\varepsilon(q)$ going to zero as $q$ goes to infinity.
We deduce from \eqref{bound1} and \eqref{as}  that  there exists a finite constant $C$ independent of $q$ (but dependent on $L$ and $k$) such that
\begin{align}
\T^{V,P}_N\left(\{\hat{\mu}_{L_N}\in B_\mu(\delta)\}\cap \mathcal K_L\right) &\geq \frac{e^{   -N C q^{-1/2k} }}{Z^V_{N,q}} \int_{K_{A}\cap \mathcal K_{L -\varepsilon(q)}^q} \mathbf{1}_{\hat{\mu}_{L_N^q}\in B_\mu(\delta-4/q)} e^{-\tr(V(L_N^q))}d{\T}^{P}_N, 
\end{align}
Since $L_N^q$ is independent of the entries $b_i, i\in J$ and therefore of $K_A$, we see that we can integrate the indicator function of $K_A$ yielding a contribution $C_A^{k_N}$ for some positive constant $C_A$ depending only on $A$. 
We observe as well that $L_N^q$ is a block diagonal matrix $\mbox{diag}(L_q^1,\ldots,L_q^{k_N},B)$ {where $L_q^i$, $1\leq i \leq k_N$, are independent and independent from $B$, $L_q^i$ following $\tilde\T_q^P$ defined in \eqref{def T tilde} and $B$ following $\tilde\T_{r_N}^P$.} Finally, we notice that $\mathcal K_{L -\varepsilon(q)}^q$ contains $\cap_{{1\leq i \leq k_N}} \{\frac{1}{q}\tr((L_q^i)^{2k}) \le L -\varepsilon(q)\}\cap \{\frac{1}{ N-k_N q}\tr(B^{2k})\le L -\varepsilon(q)\}$ since the trace of $(L_N^q)^{2k}$ is a linear combination of the latter traces. 
Thus by independence of the matrices  $L^1_q, \ldots, L^{k_N}_q$ under $\frac{1}{Z^V_{N,q}}e^{-\tr V(L^q_N)}d{\T}^{P}_N$ and convexity of balls, we deduce by taking the logarithm 
that if we set  $u_N(\delta,L) = -\ln \T^{V,P}_N(\{\hat{\mu}_{M_N} \in B_\mu(\delta)\} \cap \mathcal K_L)$  and $v_N(\delta,L) = -\ln \tilde \T^{V,P}_N(\{\hat{\mu}_{\tilde L_N} \in B_\mu(\delta)\} \cap \{\tr (\tilde L_N)^{2k})\le LN\})$, then we have
\begin{equation}
u_N(\delta + 4/q, L+\varepsilon(q) ) \leq N {(C q^{-1/2k}+\ln(C_A)/q)} + k_N v_q(\delta,L ) + v_{r_N}(\delta,L).
\end{equation}
We conclude as in Lemma \ref{wldp} that
\begin{equation}
\limsup_N \frac{u_N(\delta + 4/q, L+\varepsilon(q) )}{N} \leq \frac{v_q(\delta, L)}{q}+{Cq^{-1/2k}+ \frac{\ln(C_A)}{q}}\,.
\end{equation}
We then notice that for all $N,\delta$, $u_N(\delta,L)\ge u_N(\delta,\infty)$ and $v_N(\delta,L)\le v_N(\delta,\infty)+ \ln 2$ for $L$ large enough by Lemma \ref{exptight} (for $\tilde L_{N}$). If therefore we choose a subsequence $q$ going to infinity along which the liminf is taken, we deduce by Lemma \ref{expap} that
$$\limsup_N \frac{u_N(2\delta , \infty)}{N} \leq \liminf_{q\rightarrow\infty} \frac{v_q(\delta, \infty)}{q}=  \liminf_{q\rightarrow\infty} \frac{u_q(\delta, \infty)}{q}$$
If there is no such subsequence then both sides go to infinity and there is nothing to say. Otherwise we conclude as in Lemma \ref{wldp}. \\
{We see that we can adapt in the same fashion  the proof of Theorem \ref{ldpgen} (which stands for quadratic $V$) to our setting and get that $\mu \mapsto T^V_P(\mu)$ is convex, which concludes the proof.}
\end{proof}
\subsection{Convergence of the free energy and large deviation principle}
In the case where $V(x)=ax^{2k}, a>0$, {Lemmas \ref{exptg} and \ref{wldpT} of} the previous two sections showed that a large deviation principle holds for the empirical measure of the eigenvalues of $L_N$ under $\T^{V,P}_N$ with good{, convex} rate function {which, using \cite[Theorem 4.5.10]{DZ}, can be represented as}
\begin{equation}\label{ent}T^V_P(\mu)=-\inf_{W \in C^0_b}\{\int Wd\mu+F^{V+W,P}_\T-F^{V,P}_\T\}
\end{equation} 
where
$$F^{V,P}_\T=\lim_{N\rightarrow \infty}\frac{1}{N}\ln \int e^{-\tr V(L_N)} d\T^{P}_N\,.$$
To identify $T^V_P$ and its minimizer,  our goal is to show that
\begin{lemma} For $a>0$ and $V(x)=a x^{2k} +U(x)$ with $U \in C^0_b(\mathbb R)$, for every $P>0$, we have
\begin{equation}\label{fe}\int_{0}^{1}F^{V,sP}_\T ds= F_{C}^{V,P}
\,.\end{equation}
As a consequence, the unique minimizer of $T^V_P$ is given by $\nu_P^V=\partial_P(P\mu_P^V)$ with $\mu^V_P$ the equilibrium measure for the $\beta$-ensemble with parameter $\beta=2P/N$.
\end{lemma}

\begin{proof}
We first prove \eqref{fe}. Clearly, for all bounded continuous functions $U,U'$, uniformly in $P$,
$$|F^{a x^{2k}+U,P}_{\T}-F^{a x^{2k}+U',P}_{\T}|\le\|U-U'\|_{\infty}\mbox{ and } |F^{a x^{2k}+U,P}_{C}-F^{ax^{2k}+U',P}_{C}|\le\|U-U'\|_{\infty}\,.$$
Therefore it is enough to prove \eqref{fe} for $U\in C^{1}_{b}(\mathbb R)$ by density. We prove that for $U\in  C^{1}_{b}(\mathbb R)$,
\begin{equation}\label{der1}F^{V,P}_{\T}=\partial_{P}(PF^{V,P}_{C})\,.\end{equation}
Let us consider the tridiagonal matrix $C_P^N$ of the Coulomb model  with distribution $\P^{\frac{2P}{N}}_N$. We decompose, for $\epsilon>0$, this matrix as 
$$C_P^N=\left(\begin{array}{cc}
M_P^{\lfloor \epsilon N\rfloor} &R_N\cr
R_N^T&C_{P_N^\epsilon}^{N_\epsilon}\end{array}\right)$$
where $M_P^{\lfloor \epsilon N\rfloor} $ is a $\lfloor N\epsilon\rfloor \times \lfloor N\epsilon\rfloor$ tri-diagonal symmetric matrix with standard independent Gaussian  variables on the diagonal and chi distributed variables above the diagonal with parameters $2\frac{i}{N}P, N-\lfloor \epsilon N\rfloor\le i\le N-1$, $C_{P_N^\epsilon}^{N_\epsilon}$ is a $N_\epsilon= N-\lfloor \epsilon N\rfloor$ square tridiagonal Coulomb matrix with parameter $2P_N^\epsilon /N$ with $P_N^\epsilon=N_\epsilon N^{-1} P= (1-\lfloor \epsilon N\rfloor/N)P$, 
and $R_N$ has only one non-zero entry $r$ at position $(\lfloor \epsilon N\rfloor, \lfloor \epsilon N\rfloor+1)$. Our first goal is to show that, with $V(x)=ax^{2k}+U(x)$, we have

\begin{equation}\label{eq1}
\lim_{N\rightarrow\infty} \frac{1}{\epsilon N}\ln \E[ e^{ -\tr V(M_P^{\lfloor \epsilon N\rfloor})}] = \frac{1}{\epsilon} (F_C^{V,P}-F_C^{V,P-\epsilon})+F_C^{V,P-\epsilon}\,.
\end{equation} 
We will then complete the argument by showing  that
\begin{equation}\label{eq2}
\lim_{\epsilon \downarrow 0} \lim_{N\rightarrow\infty} \frac{1}{\epsilon N}\ln \E[ e^{ -\tr V(M_P^{\lfloor \epsilon N\rfloor})}] = F^{V,P}_\T
\end{equation} 
We next turn to the proof of \eqref{eq1}. Let us denote
$$\tilde C_P^N=\left(\begin{array}{cc}
M_P^{\lfloor \epsilon N\rfloor} &0\cr
0&C_{P_N^\epsilon}^{N_\epsilon}\end{array}\right)\,.$$
We now show that
\begin{equation}\label{exp1}
\tr ((C_P^N)^{{2k}})\ge \tr((\tilde C_P^N)^{{2k}})\,.
\end{equation}
Indeed, by Klein's lemma \cite[Lemma 4.4.12]{AGZ}, $B\mapsto \tr {(B^{2k})}$ is convex on the set of symmetric matrices. Moreover $\nabla \tr {(B^{2k})}=(2kB^{2k-1})_{ij}$. 
As a consequence, for any symmetric matrices $A, B$
$$\tr((A+B)^{2k})-\tr(B^{2k})\ge \tr (2kB^{2k-1}A)\,.$$
We apply the above inequality with $A=C_P^N-\tilde C_P^N$ and $B=\tilde C_P^N$ and notice that the entry $\lfloor \epsilon N\rfloor, \lfloor \epsilon N\rfloor +1$ of $(\tilde C_{P}^{N})^{2k-1}$ vanishes so that $\tr ((\tilde C_P^N)^{2k-1}(C_P^N-\tilde C_P^N))=0${, proving \eqref{exp1}.}\\
Moreover, if $U$ is $C^1_b$,

\begin{equation}
	\label{eq: estimate c1}
	|\tr(U(C_P^N))-\tr(U(\tilde C_P^N))|\le \int_0^1| \tr (U'(\alpha C_P^N+(1-\alpha)\tilde C_P^N)(C^P_N-\tilde C^P_N))|d\alpha\le \|U'\|_\infty |r|
\end{equation}

Consequently, using the independence of $r$ and $\tilde C_P^N$ and the fact that $C_U=\E[e^{+\|U'\|_\infty |r|}]$ is finite  since $r$ has sub-Gaussian distribution, we deduce  from \eqref{exp1} that
\begin{equation}\label{UB}
 \E[ e^{-\tr(V(C_P^N))}]\le \E[ e^{-\tr(V(\tilde C_P^N))+\|U'\|_\infty |r|  }]\le C_{U}  \E[ e^{-\tr(V(\tilde C_P^N))}]\,.\end{equation}
 As a consequence
 
 $$ \E[ e^{-\tr(V(C_P^N))}]\le  C_{U} \E[ e^{ -\tr V(M_P^{\lfloor \epsilon N\rfloor})}]
 \E[ e^{-\tr(V(C_{P_N^\epsilon}^{N_\epsilon}))}]$$
 which gives the desired lower bound:
\begin{equation}\label{UBd} 
\liminf_{N\rightarrow\infty} \frac{1}{N}\ln  \E[ e^{ -\tr V(M_P^{\lfloor \epsilon N\rfloor})}]
\ge F_C^{P,V}- (1-\epsilon) F_C^{P(1-\epsilon),V}\end{equation}
 where we used that Theorem \ref{David} is valid for $P_N^\epsilon\to (1-\epsilon)P$. 

To get the complementary lower bound we restrict ourselves to {$$\{ |r|\le \frac{1}{N}\}\cap\{\frac{1}{N}\tr((\tilde C_P^N)^{2k})\le M\}$$
Because of \eqref{eq: estimate c1} and applying H\"older's inequality as in the proof of Lemma \ref{expap}, we see that on this set}
$\tr(V(C_P^N))-\tr(V(\tilde C_P^N))$ goes to zero uniformly  for all $M$.
On the other hand the probability of the set $\{ |r|\le \frac{1}{N}\}$ is of order $1/N$. Again by independence we deduce that
\begin{eqnarray}\label{LB}
	\E[ e^{-\tr(V(C_P^N))}]&\ge& e^{{o(1)}}\E[1_{ {  \{ |r|\le \frac{1}{N}\}\cap\{\frac{1}{N}\tr((\tilde C_P^N)^{2k})\le M\}}   }
	e^{-\tr(V(\tilde C_P^N))}]\nonumber\\
	&\ge & e^{{o(1)}}\left(\E[
	e^{-\tr(V(\tilde C_P^N))}]-\E[1_{{\{\frac{1}{N}\tr((\tilde C_P^N)^{2k})\ge M\}}} 
	e^{-\tr(V(\tilde C_P^N))}]\right).
\end{eqnarray}

But we can show exactly as in the proof of Lemma \ref{exptg} that for $M$ large enough
$$\limsup_{N\rightarrow\infty}\frac{\E[1_{\{\tr( (\tilde C_P^N)^{2k})\ge MN\}} e^{-\tr(V(\tilde C_P^N))}]}{\E[e^{-\tr(V(\tilde C_P^N))}]}\le \frac{1}{2}\,,$$ yielding the desired lower bound and therefore \eqref{eq1}. 

To prove \eqref{eq2}, we proceed by approximation. We notice that if we denote by $D_{T}^\epsilon$ the density of the distribution of $M_P^{\lfloor \epsilon N\rfloor} $ with respect to the distribution of a Toda matrix $\tilde L_{\lfloor \epsilon N\rfloor}$ with parameter $P$ to which we removed the extreme entries at  $(1, \lfloor \epsilon N\rfloor)$ and 
$(\lfloor \epsilon N\rfloor,1)$, 
then we get
$$D_T^\epsilon= \prod_{i=1}^{N\epsilon} b_i^{-2P(\frac{i}{N})}\,.$$
Therefore
\begin{eqnarray*}
 \E[ e^{ -\tr V(M_P^{\lfloor \epsilon N\rfloor})}]&\ge& e^{-\epsilon^2 NM } \E[ e^{-\tr V(
 \tilde L_{\lfloor \epsilon N\rfloor})} 1_{-2P\sum_{i=1}^{\epsilon N} \frac{i}{N} \ln b_i\ge -\epsilon^{2} NM}]\\
 &=& e^{-\epsilon^{2} NM } \E[ e^{-\tr V(\tilde L_{\lfloor \epsilon N\rfloor})}]( 1- \tilde \T_{\lfloor N\epsilon \rfloor}^{V,P}(
-2P\sum_{i=1}^{\epsilon N} \frac{i}{N} \ln b_i\le -\epsilon^{2} NM) )\end{eqnarray*}
On the other hand
$$\{2P\sum_{i=1}^{\epsilon N} \frac{i}{N} \ln b_i\ge \epsilon^{2} NM\}\subset \{P \frac{1}{N\epsilon}\sum_{i=1}^{\epsilon N} b_i^2\ge M\}\subset \{\frac{1}{N\epsilon}\tr((\tilde L_{\lfloor N\epsilon\rfloor })^2)\ge M/P\}$$
has exponentially small probability under $\tilde \T_{\lfloor N\epsilon \rfloor}^{V,P}$ for ge enough. This shows{, using Lemma \ref{expap},} that there exists a finite constant $M$ such that
$$ \liminf_{N\rightarrow\infty}\frac{1}{N\epsilon}\ln \E[ e^{ -\tr V(M_P^{\lfloor \epsilon N\rfloor})}]\ge F^{V,P}_\T +M\epsilon$$
Similarly, we can see that the density $\tilde D_{T}^\epsilon= \prod_{i=1}^{N\epsilon} b_i^{2P(\frac{i}{N}-\epsilon)}$ of the law  a Toda matrix $\tilde L_{\lfloor \epsilon N\rfloor}$
with respect to  $M_P^{\lfloor \epsilon N\rfloor} $ is bounded below by $-\epsilon^{2} NM $ on  $\{ \sum_{i=1}^{\epsilon N}(\epsilon- \frac{i}{N}) \ln b_i\le\epsilon^2 NM\}$ so that we get similarly a finite constant $M'$ such that 
\begin{equation}\label{ineq2}\limsup_{N\rightarrow\infty}\frac{1}{N\epsilon}\ln \E[ e^{ -\tr V(M_P^{\lfloor \epsilon N\rfloor})}]\le  F^{V,P(1-\epsilon)}_\T+M'\epsilon\end{equation}
We hence conclude by the continuity of $\epsilon\rightarrow F^{V,P(1-\epsilon)}_\T$ (which is due to its convexity)
Equality \eqref{der1} follows then from \eqref{ineq2}.

We finally show that \eqref{fe} implies that $T^{V}_{P}$ achieves its minimum value at $\partial_{P}(P\mu^{V}_{P}).$
Indeed, by \eqref{ent}, for any bounded continuous $U$,  any probability measure $\nu$, we have
$$T^{V}_{P}(\nu)\ge -\left(\int Ud\nu +F^{V+U,P}_\T-F^{V,P}_\T\right)$$
We integrate this inequality at $\nu=\nu_{{sP}}$ a measurable probability measure valued process such that $\mu=\int_{0}^{1}\nu_{{sP}}
ds$ to deduce from \eqref{fe} that
$$\int_{0}^{1}T^{V}_{P}(\nu_{sP})ds\ge -\left(\int Ud\mu +F^{V+U,P}_C-F^{V,P}_C\right)\,.$$
We finally optimize over $U$ to conclude that
$$\int_{0}^{1}T^{V}_{P}(\nu_{sP})ds\ge -\inf_{U}\left(\int Ud\mu +F^{V+U,P}_C-F^{V,P}_C\right)=I^{V}_{P}(\mu)\,.$$
Since $I^{V}_{P}$ vanishes only at $\mu^{V}_{P}$ we deduce that any measurable minimizing path $(\nu_{sP})_{0\le s\le 1}$ must satisfy $\int_{0}^{1}\nu_{sP}ds=\mu^{V}_{P}$. {If we can consider a continuous $s\mapsto \nu_{sP}$, we conclude that $\partial_{P}(P\mu^{V}_{P})$ makes sense and that it is equal to $\nu_P$. We therefore now show that such a path can be chosen to be continuous.}
But we can follow arguments similar to those of Corollary \ref{continuite minimiseur} to show that the set $M_P^V$ where $T_P^V$  achieves its minimum value is a compact convex subset of $\mathcal P(\mathbb R)$ and is continuous in the sense that for any $\varepsilon>0$, there exists $\delta_\varepsilon>0$ such that for all $\delta<\delta_\varepsilon$,  any ${P,Q >0}$ such that  for $|{P-Q}|\le \delta$
$$M_Q^V\subset (M_P^V)^\varepsilon\,.$$
Indeed, even if we do not have the coupling of Corollary \ref{continuite minimiseur}, we easily see that the density of $\T^{V,Q}_{N}$ with respect to $\T^{V,P}_{N}$ is bounded by $e^{MN|P-Q|}$ with probability greater than $1-e^{-c(M)N}$ with $c(M)$ going to infinity when $M$ goes to infinity. Indeed, the density equals $(P-Q)\sum \ln b_i$ from which the remark follows from Lemma \ref{exptlog}. 
 This implies that
$$-\inf_{( (M_P^V)^\varepsilon)^{c}}T^{V}_{Q}\le \max\{ M|Q-P| -\inf_{( (M_P^V)^\varepsilon)^{c}}T^{V}_{P}, -c(M) N\}$$
which implies that for  any $\varepsilon>0$, for $M$ large enough and $|Q-P|$ small enough $\inf_{( (M_P^V)^\varepsilon)^{c}}T^{V}_{Q}>0$, from which the continuity follows. \end{proof}

\section{Proof of Theorem \ref{theoldp} and \ref{generalization}}
\label{proof: finale}
{Lemma \ref{wldpT} combined with the exponential tightness of Lemma \ref{exptg}}  proves a large deviation principle for the potential $V(x)=a x^{2k}$. 
If now we consider the case where $V(x)/x^{2k}$ goes to $a>0$ at infinity, we can always write $V(x)=ax^{2k}+U(x)$ where $U(x)/x^{2k}$ goes to zero at infinity. We have seen {by Lemma \ref{exptg}} that under $\T^{P,V}_N$, the event 
{$\{\frac{1}{N}\tr(L_N^{2k})>M \}$ has exponentially small probability.}
Let for $\epsilon>0$, $V_\epsilon(x)= a x^{2k} +(1+\epsilon x^{2k})^{-1}U(x)$. Then, the large deviation principle for the distribution of $\hat\mu_{L_N}$ under $\T^{V_\epsilon,P}_N$ follows from Varadhan's lemma.
Moreover,  on $\{\tr(L_N^{2k})\le M N\}$, if $|U(x)|\le \delta x^{2k}$ on $|x|\ge L$, 
\begin{eqnarray*}
\left|\frac{1}{N}\tr V(L_N)-\frac{1}{N}\tr V_\epsilon (L_N)\right|&\le& \frac{\epsilon L^{2k}}{1+\epsilon L^{2k}}\max_{|x|\le L} |U(x)| + \delta \epsilon \frac{1}{N}\tr( \frac{L_N^{4k}}{1+\epsilon L_N^{2k}})\\
&\le&  \frac{\epsilon L^{2k}}{1+\epsilon L^{2k}}\max_{|x|\le L} |U(x)| + M \delta\end{eqnarray*}
which is as small as wished if $M$ is fixed,  $L$ taken large so that $\delta$ is small, provided $\epsilon$ is taken small enough. This shows that we can approximate $\T^{V,P}_N$ by $\T^{V_\epsilon,P}_N$ in the exponential scale from which the result follows. 

The proof of Theorem \ref{generalization} follows the same arguments than those developed in the last section: we approximate the general variance profile by a stepwise constant profile, remove a negligible number of off diagonal entries and then use the large deviation principle for the Toda matrices. We leave the details to the reader.

\bibliographystyle{amsplain}

\bibliography{bibAbel3}

\providecommand{\bysame}{\leavevmode\hbox to3em{\hrulefill}\thinspace}
\providecommand{\MR}{\relax\ifhmode\unskip\space\fi MR }
\providecommand{\MRhref}[2]{%
  \href{http://www.ams.org/mathscinet-getitem?mr=#1}{#2}
}
\providecommand{\href}[2]{#2}
\begin{thebibliography}{10}

\bibitem{AGZ}
G.W. Anderson, A.~Guionnet, and O.~Zeitouni, \emph{An introduction to random
  matrices}, Cambridge Studies in Advanced Mathematics, vol. 118, Cambridge
  University Press, Cambridge, 2010. \MR{2760897}

\bibitem{DZ}
A.~Dembo and O.~Zeitouni, \emph{{Large deviations techniques and
  applications}}, {Stochastic Modelling and Applied Probability}, vol.~38,
  Springer-Verlag, Berlin, 2010, Corrected reprint of the second (1998)
  edition. \MR{2571413 (2011b:60094)}

\bibitem{dued}
I.~Dumitriu and A.~Edelman, \emph{Matrix models for beta ensembles}, J. Math.
  Phys. \textbf{43} (2002), 5830--5847.

\bibitem{Garcia}
D.~Garc\'{\i}a-Zelada, \emph{A large deviation principle for empirical measures
  on {P}olish spaces: application to singular {G}ibbs measures on manifolds},
  Ann. Inst. Henri Poincar\'{e} Probab. Stat. \textbf{55} (2019), no.~3,
  1377--1401. \MR{4010939}

\bibitem{GravaMazzuca}
T.~Grava and G.~Mazzuca, \emph{Generalized gibbs ensemble of the
  ablowitz--ladik lattice, circular $\beta$-ensemble and double confluent heun
  equation}, Communications in Mathematical Physics \textbf{399} (2023), no.~3,
  1689--1729.

\bibitem{GuioFlour}
A.~Guionnet, \emph{Large random matrices: lectures on macroscopic asymptotics},
  Lecture Notes in Mathematics, vol. 1957, Springer-Verlag, Berlin, 2009,
  Lectures from the 36th Probability Summer School held in Saint-Flour, 2006.
  \MR{2498298}

\bibitem{mazzuca2022mean}
G.~Mazzuca, \emph{On the mean density of states of some matrices related to the
  beta ensembles and an application to the toda lattice}, Journal of
  Mathematical Physics \textbf{63} (2022), no.~4.

\bibitem{Spohn2}
H.~Spohn, \emph{Ballistic space-time correlators of the classical toda
  lattice}, J. Phys. \textbf{A53} (2020), 265004.

\bibitem{Spohn3}
\bysame, \emph{The collision rate ansatz for the classical toda lattice}, Phys
  Rev. \textbf{E101} (2020), 060103.

\bibitem{Spohn1}
\bysame, \emph{Generalized {G}ibbs {E}nsembles of the {C}lassical {T}oda
  {C}hain}, J. Stat. Phys. \textbf{180} (2020), no.~1-6, 4--22. \MR{4130980}

\bibitem{Spohn4}
\bysame, \emph{Hydrodynamic equations for the toda lattice}, ArXiv 2101.06528
  (2021).

\bibitem{zhang}
D.~Zhang, \emph{Tridiagonal random matrix: {G}aussian fluctuations and
  deviations}, J. Theoret. Probab. \textbf{30} (2017), no.~3, 1076--1103.
  \MR{3687250}

\end{thebibliography}

\end{document}